\documentclass[11pt]{article}

\usepackage[a4paper, top=25mm, left=15mm, right=15mm, bottom=30mm,
headsep=10mm, footskip=12mm]{geometry}

\usepackage{amsmath}
\usepackage{amssymb}
\usepackage{amsthm}
\usepackage{amsfonts}
\usepackage{dsfont}
\usepackage[english]{babel}
\usepackage{bbm}
\usepackage{enumerate}
\usepackage{graphicx}
\usepackage{float}
\usepackage{tabularx}
\usepackage{tikz}
\usepackage{url}
\usepackage{xfrac}
\allowdisplaybreaks
\usepackage{mleftright}\mleftright

\newcolumntype{x}[1]{!{\centering\arraybackslash\vrule width #1}}

\newtheorem{theorem}{Theorem}

\newtheorem{remark}{Remark}
\newtheorem{corollary}{Corollary}
\newtheorem{definition}{Definition}
\newtheorem{conjecture}{Conjecture}
\newtheorem{proposition}{Proposition}

\newtheoremstyle{consequenceStyle}  
  {6pt}                             
  {6pt}                             
  {\itshape}                        
  {}                                
  {\bfseries}                       
  {.}                               
  { }                               
  {\thmname{#1}\thmnote{ \textbf{of Conjecture #3}}} 

\theoremstyle{consequenceStyle}
\newtheorem{consequence}{Consequence}

\newcommand{\N}{{\mathbb N}}
\newcommand{\R}{{\mathbb R}}
\newcommand{\C}{\mathbb C}
\newcommand{\Z}{{\mathbb Z}}
\newcommand{\E}{{\mathbb E}}
\renewcommand{\P}{{\mathbb P}}
\DeclareMathOperator\tr{tr}

\newcommand*\diff{\mathop{}\!\mathrm{d}}
\newcommand{\1}{\mathds{1}}
\newcommand{\eqd}{\stackrel{d}{=}}

\makeatletter
\let\@fnsymbol\@arabic
\newcommand{\specificthanks}[1]{\@fnsymbol{#1}}
\makeatother

\begin{document}
\mleftright

\title{\textbf{Random eigenvalues of graphenes and the triangulation of plane}}
\insert\footins{\footnotesize \noindent Artur Bille \hspace*{3cm} Simon Coste \\
\texttt{artur.bille@uni-ulm.de} \hspace*{0.9cm} \texttt{simon.coste@u-paris.fr}  \vspace*{0.2cm} \\
Victor Buchstaber \hspace*{2cm} Satoshi Kuriki \hspace*{2cm} Evgeny Spodarev\\ 
\texttt{buchstab@mi-ras.ru} \hspace*{1.6cm} \texttt{kuriki@ism.ac.jp}\hspace*{1.5cm} \texttt{evgeny.spodarev@uni-ulm.de}  \vspace*{0.2cm}\\
\textsuperscript{1}Ulm University, Ulm, Germany \quad
\textsuperscript{*}Corresponding author\quad
\textsuperscript{2}Steklov Mathematical Institute RAN, Moscow, Russia\\
\textsuperscript{3}HSE University, International Laboratory of Algebraic Topology and Its Application, Moscow, Russia \\
\textsuperscript{4}LPSM - Université Paris-Cité, Paris, France \quad
\textsuperscript{5}The Institute of Statistical Mathematics, Tokyo, Japan
}

\author{Artur Bille\textsuperscript{1*} 
\and 
Victor Buchstaber\textsuperscript{2,3}\and 
Simon Coste\textsuperscript{4}\and 
Satoshi Kuriki\textsuperscript{5}\and
Evgeny Spodarev\textsuperscript{1}
}
\date{}
\maketitle

\begin{abstract}
\label{abstract}
\noindent
We analyze the numbers of closed paths of length $k\in\N$ on two important regular lattices: the hexagonal lattice (also called {\it graphene} in chemistry) and its dual triangular lattice. 
These numbers form a  moment sequence of specific random variables connected to the distance of a position of a planar random flight (in three steps) from the origin. 
Here, we refer to such a random variable as a \textit{random eigenvalue} of the underlying lattice. 
Explicit formulas for the probability density and characteristic functions of these random eigenvalues are given for both the hexagonal and the triangular lattice. 
Furthermore, it is proven that both probability distributions can be approximated by a functional of the random variable uniformly distributed on increasing intervals $[0,b]$ as $b\to\infty$. 
This yields a straightforward method to simulate these random eigenvalues without generating graphene and triangular lattice graphs. 
To demonstrate this approximation, we first prove a key integral identity for a specific series containing the third powers of the modified Bessel functions $I_n$ of $n$th order, $n\in\Z$. 
Such series play a crucial role in various contexts, in particular, in analysis, combinatorics, and theoretical physics. 
\end{abstract}


\section{Introduction}
\label{sec: Introduction}
In 2010, Sir Konstantin Novoselov and Sir Andre Geim were awarded the Nobel prize in Physics for their method to isolate single layers of \emph{graphene}. 
Graphene is a carbon allotrope in which the carbon atoms are arranged in a cut-out of an infinitely large hexagonal lattice. 
It can be considered as an extreme case of other finite-sized carbon allotropes like \emph{fullerenes}. 
The analysis of graphene is a mathematical problem with a long history. 
The hexagonal lattice is one of the well-known two-dimensional \textit{Bravais lattices}, cf. \cite[1.2.5.4]{wondra}. 
Although many (mathematical) results about graphene and its dual, the triangular lattice, are known and may be found, e.g., in \cite{grunbaum1987tilings,coxeter1973regular}, many more questions are still open. 
The similarities and differences between graphene and its finite-sized counterparts are also of great interest. 

In this paper, we study the spectral properties of these carbon allotropes. 
The spectral density of a lattice is linked with its combinatorial properties, like the numbers of closed paths rooted at some vertex. 
Indeed, let $\mu_k\left(\mathcal{L}\right)$ be the number of closed paths of length $k$ on the lattice $\mathcal{L}$ starting in an arbitrary vertex of $\mathcal{L}$.
For a finite graph $G_n$ with $n$ vertices, we denote by $\mu_k\left(G_n\right)$ the average number of closed paths of length $k$ on $G_n$. 
It can be easily seen that
\begin{align*}
\mu_k\left(G_n\right) = \frac{1}{n} \tr\left( A^k\right) = \frac{1}{n} \sum_{j=1}^{n} \lambda_j^k,
\end{align*} 
where $\lambda_1,\ldots,\lambda_{n}$ are the eigenvalues of the adjacency matrix $A$ describing the graph $G_n$. 
The term $\tr\left(A^k\right)$ is often referred to as the \textit{Newton polynomial of order} $k$. An in-depth analysis of Newton polynomials of order $k \leq n$ for (dual) fullerene graphs with $n \leq 150$ can be found in \cite{Bille2020SpectralCO}.
Therefore, $\mu_k(G_n)$ is the $k^{\text{th}}$ moment of a probability measure $\varrho_{G_n}$ with empirical distribution function given by
\begin{align}\label{def:esd}
F_{\varrho_{G_n}}(x)= \frac{1}{n}\sum_{j=1}^{n} \1\lbrace \lambda_j \leq x \rbrace. 
\end{align}
Since each eigenvalue $\lambda_j$ is chosen with uniform probability, $\varrho_{G_n}$ is the distribution of a random variable $L_n$ which will often be referred to as \textit{random eigenvalue} of $G_n$. 
The measure $\varrho_{G_n}$ is often called \emph{empirical spectral distribution} (ESD) of $G_n$. 
For lattices $\mathcal{L}$, i.e. infinite graphs with all vertices having a finite degree, one cannot directly define its spectral distribution as in \eqref{def:esd}, but one can still define $\varrho_{\mathcal{L}}$ through its moments: 
\begin{align}\label{def:DOS}
\int_\R x^k d\varrho_{\mathcal{L}}(x) = \mu_k\left(\mathcal{L} \right),\quad k\in\N.
\end{align} 
The measure $\varrho_{\mathcal{L}}$ is well-known as the \emph{density of states} or \emph{spectral density} of $\mathcal{L}$; it is a probability distribution. Any random variable with law $\varrho_{\mathcal{L}}$ will be called \emph{random eigenvalue} of $\mathcal{L}$. 

In both cases, the spectral density of a structure captures many properties of the geometry of the underlying graph or lattice. It is our primary goal to study the spectral properties of graphene. 
In this paper, we shed light on how the ESD $\varrho_{F_n}$ of large (but finite)  fullerenes $F_n$ approximates the density of states of graphene, $\varrho_{\mathcal{L}}$, and to analyze $\varrho_{\mathcal{L}}$ itself. 
In particular, we use a notion of graph convergence to show how graphene can be seen as an infinitely large fullerene, and the consequences on their respective spectral densities. 
Mathematically, the convergence of random fullerenes towards graphene is an open question (Conjecture \ref{conjecture}). 
The main part of the paper then deals with the spectral properties of $\mathcal{L}$. 
Although $\varrho_{\mathcal{L}}$ has been studied in the physics literature mostly under the lens of $\mathcal{L}$'s Green function, cf. \cite{horiguchi1972lattice}, its explicit expression is not commonly encountered in the literature. 
In this work, we show several identities and probabilistic representations for $\varrho_{\mathcal{L}}$, involving elegant and striking integral formulas for cylinder functions. 

The paper is structured as follows. In the next section, we formally define lattices and graphs under consideration, we also recall the notion of local weak convergence for graphs, and its consequences for the spectrum. 
Section \ref{sec:main results} contains our main results followed by their discussion. 
There, we propose a novel integral representation of the series of third powers of modified Bessel functions (Theorem \ref{theorem:Besselfunction}) which is used
to prove a simple approximation of the distribution of a random eigenvalue of the triangular and hexagonal lattices (Theorem \ref{theorem:approximation}). 
We give an independent proof of Theorem \ref{theorem:approximation} elucidating some connections to ergodic theory. 
Furthermore, we give explicit forms of the densities and the characteristic functions of random eigenvalues for graphene and the triangulation of plane.
Section \ref{sec:conjecture} explains the point of view of local weak convergence of fullerenes and gives some details on Conjecture \ref{conjecture}. 
Proofs are given in Section \ref{sec:proofs}.
The Python code supporting the results of this paper can be downloaded from \cite{RsoftBille23}.

\section{Spectral properties of graphenes and fullerenes}\label{sec:preliminaries}
In this section, we collect known definitions and results on Bessel functions, regular planar lattices, paths on them as well as planar random flights. In the sequel, $\N_0$ denotes the set of natural numbers and zero.

\subsection{Lattices and planar graphs}
A \textit{graph} $G=(V(G),E(G))=(V,E)$ is a tuple consisting of sets of vertices $V$ and of edges $E\subseteq V^2$. 
If $\left|V\right|=\infty$ and every vertex has a finite degree the graph $G$ is usually called an \textit{infinite graph}. 
An introduction to the general theory of infinite graphs can be found in \cite{diestel_graphtheory}. 
Here, we consider only infinite non--oriented graphs whose planar embedding forms a tiling of the plane $\R^2$ by convex regular polygons. 
We call such graphs \textit{(convex) lattices}. 

Let us define two lattices whose random eigenvalues are studied in the sequel. 
\begin{figure}[H]
  \begin{center}
    \includegraphics[scale=1]{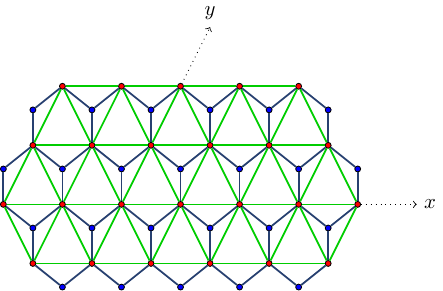}
    \caption{Illustration of the relationship between $\mathcal{H}$ and 
        $\mathcal{T}$.
        The set $V(\mathcal{H})$ consists of both blue and red-colored vertices, whereas $V(\mathcal{T})$ contains only the red-colored ones.
        Edges of $\mathcal{H}$ and $\mathcal{T}$ are colored in black and green, respectively.}\label{fig:T_H}
  \end{center}
\end{figure}
\begin{definition}\label{def:H_T_N_N*}
    We call the infinite graph 
    \begin{enumerate}[(i)]
        \item $\mathcal{H}$ a \textit{hexagonal lattice},
                if its set of vertices is given by
        \begin{align}\label{eq:V(h)}
            V(\mathcal{H})
            :=
            \left\lbrace \left(\sqrt{3} x + \frac{y\sqrt{3}}{2},\frac{3y}{2} + c\right)^{\intercal}\in\R^2 ~\middle|~ x,y\in\Z,  ~c\in\lbrace 0,1 \rbrace \right \rbrace,
        \end{align}
        and every vertex is connected with its three nearest neighbors (w.r.t.\ the Euclidean distance) by an edge;
        \item $\mathcal{T}$ a \textit{triangular lattice}, if its set of vertices is given by
        \begin{align*}
            V(\mathcal{T})
            :=
            \left\lbrace v\in V(\mathcal{H})~\middle|~ c=0 \right \rbrace
        \end{align*}
        and every vertex is connected with its six nearest neighbors (w.r.t.\ the Euclidean distance) by an edge. 
    \end{enumerate}
\end{definition}
Roughly speaking, the hexagonal lattice is composed of two triangular sublattices.
Adding three loops to every vertex of $\mathcal{T}$ and denoting this modification by $\mathcal{T}^*$, yields a simple bijection between the set of closed paths with even length $2k$ on $\mathcal{H}$ and the set of closed paths with length $k$ on $\mathcal{T}^*$. These three loops can be mathematically formulated as one loop of weight $3$.

\bigskip

Let us now turn to formal definitions of fullerenes. 
It is well known that Euler's formula and Eberhard's theorem impose hard constraints on the structure of finite planar graphs: 
for example, there can be no finite planar graph which is 3-regular and with only hexagonal faces. One has to introduce faces with degree smaller than 6 to fulfill these constraints; in fullerenes, one only allows faces with degree 5 or 6 and it turns out that the number of faces of degree 5 (called pentagons) must exactly be $12$. 

\begin{definition}
    A fullerene (graph) is a finite, connected, 3-regular planar graph with faces of degrees both 5 and 6. 
\end{definition}

\subsection{Fullerenes and the planar hexagonal lattice}\label{subsec:lwc}
Local weak convergence, also called Benjamini-Schramm convergence due to the seminal paper \cite{benjamini2011recurrence}, is a mode of convergence well suited to graphs with a small number of edges. 
We refer to \cite{benjaminicurien12,aldous2004objective,bordenave2016spectrum} for in-depth presentations; for our purposes, let us simply recall the basic elements. 

If $H=(V,E)$ is a graph and $v\in V$, we denote by $B_H(v,r)$ the ``ball of radius $r$ around $v$'', that is, the subgraph of $H$ induced by all the vertices of $H$ whose graph distance to $v$ is smaller than $r$. 
A graph $H$ endowed with a special vertex $o$ is called a \emph{rooted graph}; the set of all locally finite rooted graphs can be endowed with a Polish space structure. 
We say that (the law of) a sequence of finite random graphs $(G_n)$ converges locally weakly towards (the law of) a random rooted graph $(G,o)$ if for any fixed integer $r$ and fixed rooted graph $(H,o)$, 
\begin{equation}\label{def:BSConv} 
    \P(B_{G_n}(o_n, r) \text{ is isomorphic to } B_H(o,r)) 
    \xrightarrow[n\to\infty]{} 
    \P(B_{G}(o, r) \text{ is isomorphic to } B_H(o,r)),
\end{equation}
where $o_n$ is itself a random element uniformly distributed over the vertices of $G_n$. 
This notion of convergence indeed coincides with the local weak convergence of probability measures on the space of rooted graphs mentioned above. 
The limiting random rooted graph $(G,o)$ can be infinite and always has a property called \emph{unimodularity} which is very restrictive, see \cite{benjaminicurien12}. 
The convergence of random planar graphs has recently been widely studied, especially for uniform random triangulations of the sphere. 
However, to our knowledge, the following statement remains open. 

\begin{conjecture}\label{conjecture} 
    Let $F_n$ be a random fullerene, chosen uniformly on the set of fullerenes with $n$ vertices. 
    The sequence $(F_n)$ converges locally weakly towards the planar hexagonal lattice $\mathcal{H}$. 
\end{conjecture}

We will comment more thoroughly on this question in Section \ref{sec:conjecture}. 
Conjecture \ref{conjecture} has noticeable consequences on the behavior of the eigenvalues of typical fullerenes. 
Indeed, the convergence of a sequence of finite graphs $(G_n)$ towards a random rooted graph $(G,o)$ implies the convergence of the empirical spectral distribution $\varrho_{G_n}$ towards a limiting probability measure $\varrho$, see \cite{abert2013benjamini} or \cite{bordenave2016spectrum}. 
This measure $\varrho$ can be defined directly from the law $(G,o)$. 
The general construction is functional-analytic in essence. 
Under mild assumptions such as uniform boundedness of the degrees in $G$, the measure $\varrho$ is the unique Borel measure on an interval $I\subset \R$ whose $k$-th moments are equal to $\E \mu_k(G)$, the expected number of closed walks at the root of $(G,o)$; more precisely, $\varrho$ is the unique probability measure whose Stieltjes transform is given by
\begin{equation}\label{def:rho}
    \int_I \frac{1}{z-x}\varrho(dx) 
    = 
    \sum_{k=0}^\infty \frac{\E \mu_k(G)}{z^{k+1}},
    \quad z\in \C\setminus I. 
\end{equation}
The measure $\varrho$ is called \emph{the spectral density} of (the law of) the random graph $(G,o)$; it is also known by the name of \emph{averaged density of states} in the physics community.  
For the hexagonal lattice $\mathcal{H}$, the construction given in \eqref{def:DOS} coincides with \eqref{def:rho}, since in this case the number of closed paths at the root is nonrandom. 

The spectral density of $\mathcal{H}$ will be systematically denoted by $\varrho_{\mathcal{H}}$. 
Similarly, the spectral density of the dual lattice $\mathcal{T}^*$ will be denoted by $\varrho_{\mathcal{T}^*}$. 

The preceding discussion together with Conjecture \ref{conjecture} directly implies that the ESD of large random fullerenes converges towards the spectral density of $\mathcal{H}$. 

\begin{consequence}[1]
    Let $\varrho_{\mathcal{H}}$ be the spectral density of $\mathcal{H}$, uniquely defined by \eqref{def:rho}.  
    The empirical spectral distribution of random fullerenes with $n$ vertices converges weakly towards $\varrho_{\mathcal{H}}$ as $n\to\infty$. 
\end{consequence}
Our goal in the sequel will be to thoroughly study the spectral density $\varrho_{\mathcal{H}}$ and its various representations. 
We will start by a careful examination of the sequence $\mu_k(\mathcal{H})$. 

\subsection{Paths on the planar hexagonal lattice }
\label{random walks on the hexagonal lattice and in the plane}
Let us discuss the computation of the number of paths on $\mathcal{H}$ between two vertices with a given number of steps $k$. 
Without loss of generality, let the starting point be the origin and the endpoint a vertex $v\in V(\mathcal{H})$. 
Finally, set $v=(0,0)$ to get the number of closed paths of length $k$. 
Clearly, it is sufficient to consider even lengths $2k$ with arbitrary $k\in\N_0$ only. 
In the sequel, we follow \cite{kasteleyn1967graph} in our exposition to prove the following explicit expression for the numbers $\mu_{2k}(\mathcal{H})$. 
\begin{theorem}[\!\!\cite{Cotfas00}, Theorem 3]\label{theorem:cotfas_integral}
    For $k\in \N_0$ it holds
    \begin{align}
        \mu_{2k}\left( \mathcal{H} \right)
        =
        \frac{1}{(2\pi)^3} 
        \int_{0}^{2\pi} \int_{0}^{2\pi}\int_{0}^{2\pi} 
        \left| e^{i \varphi_1}+e^{i \varphi_2}+e^{i \varphi_3} \right|^{2k}  
        \diff \varphi_1 \diff \varphi_2 \diff \varphi_3. \label{eq:closed_walks_integral}
    \end{align}
\end{theorem}
\begin{proof}
    Every path on $\mathcal{H}$ can be described as a sequence of the three direction vectors
    \begin{align*}
    x_1=\left( 0,1 \right)^\intercal, \quad x_2=\left( \frac{1}{2},-\frac{\sqrt{3}}{2} \right)^\intercal, \quad x_3=\left( -\frac{1}{2},-\frac{\sqrt{3}}{2} \right)^\intercal,
    \end{align*}
    and their inverse vectors
    \begin{align*}
        x_1^{-1} = (0,-1)^\intercal, \quad
        x_2^{-1} = \left( -\frac{1}{2},\frac{\sqrt{3}}
                   {2}\right)^\intercal, \quad 
        x_3^{-1} = \left( \frac{1}{2},\frac{\sqrt{3}}
                   {2}\right)^\intercal,
    \end{align*}
    compare Figure \ref{fig:T_H}. 
    Thus, a path of length $2k$ starting at the origin and ending in $v$ can be written as a sequence 
    \begin{align}
        x_{i_1}x_{j_1}^{-1}x_{i_2}x_{j_2}^{-1}\ldots 
        x_{i_k}x_{j_k}^{-1} 
        \label{sequence_of_directions}
    \end{align}
    for $i_1,j_1,\ldots,i_k,j_k\in\lbrace 1,2,3 \rbrace$.
    Note that every odd step on $\mathcal{H}$ has to be $x_1,x_2$ or $x_3$. 
    On the other hand, every even step has to be one of the inverse directions $x_1^{-1},x_2^{-1}$ or $x_3^{-1}$.
    Hence, the total number of directions $x_1,x_2,x_3$ in such sequence has to be equal to the total number of inverse directions $x_1^{-1},x_2^{-1},x_3^{-1}$.
    Moreover, the specific order of directions is not important for the calculation of the number of paths.
    Hence, in (\ref{sequence_of_directions}) we can rearrange the elements, cancel directions with their inverses, and write the remaining amount of steps in direction $x_j$ for $j = 1,2,3$ as its power. 
    For every destination vertex $v$, one gets a unique minimal amount of directions one has to go, i.e., the powers $k_1,k_2,k_3\in\Z$ in the shortened form $x_1^{k_1}x_2^{k_2}x_3^{k_3}$ of (\ref{sequence_of_directions}) are uniquely determined by $v$. 
    If $v=(0,0)$ it follows immediately that $k_1=k_2=k_3=0$.
    
    One can get the number of paths terminating in $v$ as the coefficient of the term $x_1^{k_1}x_2^{k_2}x_3^{k_3}$ in the multinomial expansion of
    \begin{align}
        \left(\left(x_1+x_2+x_3\right)
        \left(x_1^{-1}+x_2^{-1}+x_3^{-1}\right)\right)^{k},
        \label{sequence_of_directions2}
    \end{align}
    or,  equivalently, as the constant term in the multinomial expansion of
    \begin{align*}
        \left(\left(x_1+x_2+x_3\right) \left(x_1^{-1}+x_2^{-1}+x_3^{-1}\right)\right)^{k} 
        x_1^{-k_1}x_2^{-k_2}x_3^{-k_3}.
    \end{align*}
    To get this constant term, replace $x_j$ by complex number $e^{i\varphi_j}$, $j=1,2,3$, and integrate the resulting expression over all possible $\varphi_1,\varphi_2,\varphi_3\in[0,2\pi]$. 
    This step is reasonable since $x_1,x_2,x_3$ and their inverses are unit vectors in $\R^2$ and thus can be interpreted as unit complex numbers.
    In other words, we transform Cartesian coordinates into polar coordinates.
    
    Now, the amount of paths of length $2k$ from the origin to $v$ is given by
    \begin{align*}
        &\frac{1}{(2\pi)^3} 
        \int_{0}^{2\pi} \int_{0}^{2\pi} \int_{0}^{2\pi} 
        \left(\left(e^{i\varphi_1} + e^{i\varphi_2} + e^{i\varphi_3}\right)
        \left(e^{-i\varphi_1} + e^{-i\varphi_2} + e^{-i\varphi_3}\right) \right)^{k} 
        e^{-i(k_1\varphi_1 + k_2\varphi_2 + k_3\varphi_3)} 
        \diff \varphi_1\diff \varphi_2 \diff \varphi_3.
    \end{align*}
    Notice that
    $ \left(\left(e^{i\varphi_1} + e^{i\varphi_2} + e^{i\varphi_3}\right)
    \left(e^{-i\varphi_1} + e^{-i\varphi_2} + e^{-i\varphi_3}\right) \right) 
    = 
    \left| e^{i\varphi_1} + e^{i\varphi_2} + e^{i\varphi_3} \right|^{2}
    $
     and set $k_1=k_2=k_3=0$ for the number of closed paths of length $2k$. 
\end{proof}

\subsection{Special functions}
Before going further, we briefly need to recall some definition and well-known results for the (modified) Bessel functions  used in the proofs of our next results.
For integers $n\in\Z$, the \textit{Bessel function of the first kind of order} $n$ can be defined by the series
\begin{align*}
    J_n(x) 
    := 
    \sum_{k=0}^\infty \frac{(-1)^k}{k! (k+n)!}\left(\frac{x}{2}\right)^{2k+n},\quad x\in\C.
\end{align*}
 The \textit{modified Bessel function of the first kind of order} $n$ is  closely related to $J_n$ by
\begin{align*}
    I_n(x) 
    := 
    i^{-n}J_n(ix)
    = 
    \sum_{k=0}^\infty \frac{1}{k! (k+n)!}\left(\frac{x}{2}\right)^{2k+n},\quad x\in\C. 
\end{align*}
Results given in \cite{abra} imply the following useful relation: 
\begin{align}
    2 \frac{\diff I_n(x)}{\diff x} 
    = 
    I_{n-1}(x) + I_{n+1}(x), 
    \quad x\in\C,\quad n\in\Z. \label{eq:derivative}
\end{align}
Another important property of the modified Bessel function is the following recursive formula, which is a direct consequence of \cite{abra}: 
\begin{align*}
    I_{n}(x) 
    =
    I_{n+2}(x)+ \frac{2(n+1)}{x} I_{n+1}(x), 
    \quad x\in\C,\quad n\in\Z.
\end{align*}
We will occasionally need the {\it Gaussian hypergeometric function}
${}_2F_1$ which is defined as 
\begin{align*}
    {}_2F_1\left(
    \begin{matrix}
        a_1,a_2\\
        b
    \end{matrix}
    ;x\right)
    :=
    \sum_{j=0}^\infty 
    \frac{a_1^{(j)} a_2^{(j)} x^j}{b^{(j)} j!}, 
    \quad x\in\R,
\end{align*} 
for  $a_1,a_2,b\in\C$
where $a^{(n)}$ is the \textit{rising factorial}, i.e., 
\begin{align*}
    a^{(n)}
    := 
    \prod_{j=0}^{n-1} \left( a+j\right).
\end{align*}

\subsection{Properties of the spectral density of $\mathcal{H}$}
In equation \eqref{eq:closed_walks_integral}, the sum $ e^{i \varphi_1}+e^{i \varphi_2}+e^{i \varphi_3}$ can be seen as the position $W_3$ of a planar random flight $W$ after three steps. 
This provides a surprising connection between the moments of the spectral density $\varrho_{\mathcal{H}}$ and the position of a certain random walk in two dimensions. 
We define the \emph{planar random flight} as an isotropic \textit{random walk} $W=\{W_t, \; t\in \N_0\}$ on $\R^2=\C$ in discrete time starting from the origin ($W_0=(0,0)$) such that $W_{t+1}=W_t+e^{i U_t}$, where the random angles $U_t\sim U[0,2\pi]$, $t\in\N$ form an i.i.d.\ sequence. 

This interpretation yields that (\ref{eq:closed_walks_integral}) is nothing but the $(2k)^{\text{th}}$ moment of $X:=|W_3|$, and in particular the spectral density $\varrho_{\mathcal{H}}$ and the probability distribution of $X$ are identical, providing a very surprising representation for $\varrho_{\mathcal{H}}$ that will be further exploited in our main results. 
From now on, we gather some results on the density of $X$.

Borwein et al. presented in \cite[Theorem 2.1]{borwein15} a formula for the density of $|W_t|$ in case of a spatial random flight $W$ in $\R^d$ for arbitrary dimension $d\ge 2$ and any number of steps $t\in\N$. 
Let us formulate their result in our specific case $d=2$, $t=3$:
\begin{theorem}
    The random variable $X$ is absolutely continuously distributed with density
    \begin{align*}
        f_{X}(x)
        =
        \int_0^{\infty} tx J_0(tx) J_0^3(t) \diff t,
        \quad x\ge 0. 
    \end{align*}
\end{theorem}

Next we analyze the behavior of $f_{X}$ on the positive real half line.

\begin{proposition}\label{pro:density_X_support}
    It holds
    \begin{align*}
        f_{X}(x) =
        \begin{cases}
            \frac{2\sqrt{3}x}{\pi(3+x^2)}{}_2F_1\left( 
            \begin{matrix} 
                \frac{1}{3},\frac{2}{3} \\
                1
            \end{matrix};
            \frac{x^2(9-x^2)^2}{(3+x^2)^3} \right),\quad &\text{ if }x\in [0,3],\\
            0,&\text{ if }x>3,
        \end{cases}
    \end{align*}
    where ${}_2F_1$ is the Gaussian hypergeometric function.
\end{proposition} 
The above formula for $x\in[0,3]$ was derived in \cite{borwein2012densities}. 
See also \cite[Theorem 4.5]{borwein15}.
If $x>3$ it follows
\begin{align*}
    \int_0^\infty t J_0(tx)J_0(t)^3 \diff t
    =
    0
\end{align*} 
by \cite{watson66} where we set $a_1=x, a_2=a_3=1$ and $\nu=0$. 
The fact that the support of $f_X$ lies in the interval $[0,3]$ is very intuitive since $X$ is a distance $0\le |W_3|\le 3$.
The even moments of $X$ were derived in \cite{borwein15}: 
\begin{proposition}\label{prop:moments_X}
    For $k\in\N_0$, it holds
    \begin{align*}
        \E X^{2k} 
        = 
        \sum_{k_1+k_2+k_3=k}\binom{k}{k_1,k_2,k_3}^2.
    \end{align*}
\end{proposition}

As an immediate consequence, the number of closed walks on graphene $\mathcal{H}$ and the triangular lattice $\mathcal{T}^*$ can be obtained if we notice that 
$\mu_k\left(\mathcal{T}^*\right) =\mu_{2k}\left(\mathcal{H}\right)$ 
and 
$\mu_{2k-1}\left(\mathcal{H}\right)=0$ for any $k\in\N$. 
The latter relation holds since any path on $\mathcal{H}$ with an odd length $k\in\N$ ends in a vertex with parameter $c=1$ in \eqref{eq:V(h)} and thus cannot be closed.
\begin{proposition}\label{prop:number of closed walks}
    The number of closed paths of length $k$ at the root of $\mathcal{H}$ are given by 
    \begin{align} \label{eq:mu_H}
        \mu_k\left(\mathcal{H}\right) &= 
        \begin{cases}
            \sum\limits_{k_1+k_2+k_3=k/2}\binom{k/2}{k_1,k_2,k_3}^2\quad &\text{if }k \text{ is even,} \\
            0 &\text{if }k \text{ is odd}.
        \end{cases} 
    \end{align}
    Similarly, the number of closed paths of length $k$ at the root of $\mathcal{T}^*$ are given by 
    \begin{align}
        \mu_k\left(\mathcal{T}^*\right) 
        &= 
        \sum\limits_{k_1+k_2+k_3=k}\binom{k}{k_1,k_2,k_3}^2. \label{eq:mu_T*}
    \end{align}
\end{proposition}

\begin{remark}
    For even values of $k$, formula \eqref{eq:mu_H} can be derived in either of the following two ways:
    \begin{enumerate}[(a)]
        \item By applying \cite{DiCre} and using Vandermonde's identity \cite{abra}, along with setting all transition probabilities of the random walk on $\mathcal{H}$ in \cite{DiCre} to $1/3$. 
        \item By employing combinatorial arguments and the multinomial theorem to Formulae \eqref{sequence_of_directions} and \eqref{sequence_of_directions2}, we obtain
        \begin{align*}
            &\left(\left(x_1+x_2+x_3\right)
            \left(x_1^{-1}+x_2^{-1}+x_3^{-1}\right)\right)^{k} \\
            &=
            \left( \sum_{k_1+k_2+k_3=k}\binom{k}{k_1,k_2,k_3} x_1^{k_1}x_1^{k_3}x_3^{k_3} \right)
            \left( \sum_{l_1+l_2+l_3=k}\binom{k}{l_1,l_2,l_3} x_1^{-l_1}x_2^{-l_2}x_3^{-l_3} \right)\\
            &= \sum_{\substack{k_1+k_2+k_3=k\\ l_1+l_2+l_3=k}} \binom{k}{k_1,k_2,k_3}\binom{k}{l_1,l_2,l_3} x_1^{k_1-l_1}x_2^{k_2-l_2}x_3^{k_3-l_3}.
        \end{align*}
        Since each direction \( x_j \) and its inverse \( x_j^{-1} \) must occur in the path with the same frequency, it follows that \( k_j = l_j \) for \( j=1,2,3 \). 
    \end{enumerate}
\end{remark}

\section{Main results}\label{sec:main results}
Here we formulate our main results:

\subsection{Distribution of random eigenvalues of graphene $\mathcal{H}$ and its dual graph $\mathcal{T}^*$}
\label{sec:Main.subsec:DistrEigenvTH}

\begin{proposition}\label{proposition:densities_for_sequences}
The probability distribution $\varrho_{\mathcal{H}}$ is absolutely continuous with respect to the Lebesgue measure, and has a density given by
\begin{align*}
    f_H(x)
    &=
    \frac{1}{2}\int_0^\infty t|x| J_0(|x|t)J_0^3(t)\diff t,
    \quad x\in\R,\\ 
    &=
    \begin{cases}
        \frac{\sqrt{3}|x|}{\pi \left(3+x^2\right)}{}_2F_1\left(
        \begin{matrix}
            \frac{1}{3},\frac{2}{3} \\
            1
        \end{matrix}; 
        \frac{x^2\left(9-x^2\right)^2}{\left(3+x^2\right)^3} \right),\quad &\text{ if }x\in [-3,3],\\
        0, &\text{otherwise.}
    \end{cases}
\end{align*}
Similarly, the density of the probability law $\varrho_{\mathcal{T}^*}$ is given by
\begin{align*}
    f_T(x)
    &=
    \frac{1}{2}\int_0^\infty t J_0(t\sqrt{x})J_0^3(t)\diff t, 
    \quad x>0,\\
    &=\begin{cases}
        \frac{\sqrt{3}}{\pi (3+x)}{}_2F_1 \left(
        \begin{matrix}
            \frac{1}{3},\frac{2}{3}\\
            1
        \end{matrix}; 
        \frac{x\left(9-x\right)^2}{\left(3+x\right)^3} \right),\quad &\text{ if } x\in [0,9],\\
        0,&\text{otherwise.}
    \end{cases}
\end{align*}
\end{proposition}

\begin{figure}[h] 
    \begin{center}
    \begin{minipage}{0.49\textwidth}
        \centering
        \includegraphics[scale=0.5]{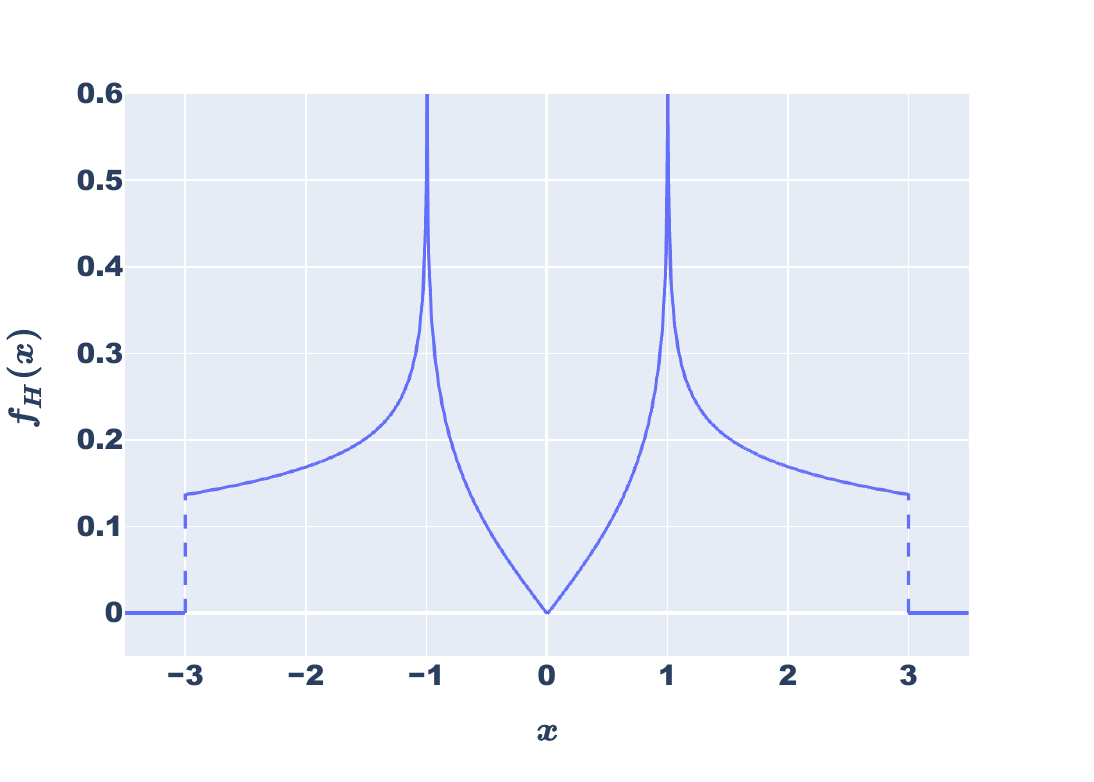}
    \end{minipage}\hfill
    \begin{minipage}{0.49\textwidth}
        \centering
        \includegraphics[scale=0.5]{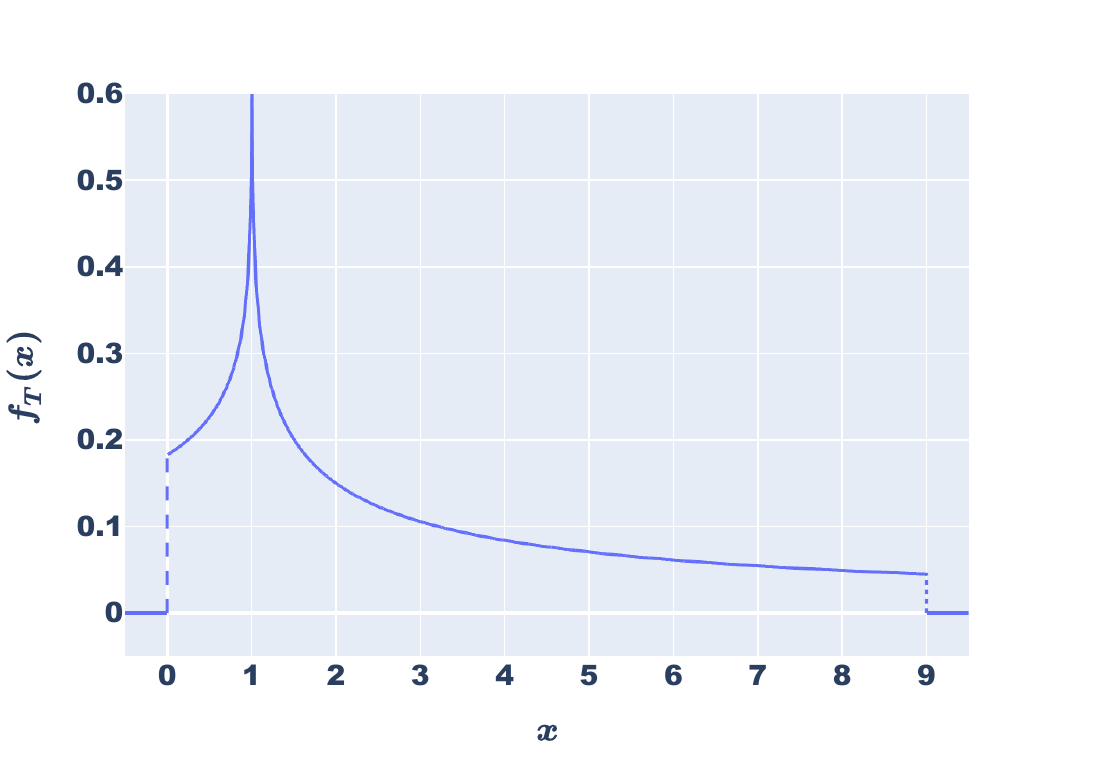}
    \end{minipage}
    \end{center}
      \caption{Density functions $f_H$ (left) and $f_T$ (right) of the random eigenvalue $H$, $T$ of the hexagonal lattice $\mathcal{H}$ and triangular lattice $\mathcal{T}^*$, respectively.}
      \label{fig:density_hexagonal_lattice}
\end{figure}

The density function $f_H$ has two logarithmic singularities in $x=\pm 1$, cf. \cite{borwein15}. 
By construction it follows that $f_T$ has a logarithmic singularity in $x=1$, as illustrated in Figure \ref{fig:density_hexagonal_lattice}. 

In the sequel, we will use $H$ and $T$ to denote random variables whose probability distribution is $\varrho_{\mathcal{H}}$ or $\varrho_{\mathcal{T}^*}$, or equivalently with densities $f_H$ or $f_T$. It can be immediately seen from the above densities that the following important relation holds:
\begin{equation}\label{eq:TdH}
    H\eqd A \sqrt{T},
\end{equation}
where $A$ is a Rademacher distributed random variable taking values $\pm 1$ with probability $1/2$ each, and $\eqd$ means the equality of probability laws. $A$ and $T$ must be chosen stochastically independent from each other. The next result provides a handy integral representation for the characteristic functions of $H$ and $T$.
\begin{proposition} \label{prop:charfctHT}
    For the characteristic functions $\varphi_H(s)=\E e^{isH}$ and $\varphi_T(s)=\E e^{isT}$, $s\in\R$, it holds
    \begin{align}
        \varphi_H(s) 
        &= 
        \frac{\diff}{\diff s} \left[ s \int_0^1 I_0^3\left( 2 is \sqrt{ t(1-t)} \right) \diff t \right], \label{eq:charFH}   \\
        \varphi_T(s) 
        &=
        \int_0^1 I_0^3\left( 2i\sqrt{is\log t} \right)\diff t\, . \label{eq:charFT}
    \end{align}
\end{proposition} 
Calculating the derivative in \eqref{eq:charFH} and using property \eqref{eq:derivative} we get
\begin{equation*}
    \varphi_H(s)
    = 
    \int_0^1 I_0^3\left( 2 is \sqrt{ t(1-t)} \right) \diff t 
     + 6is \int_0^1 \sqrt{ t(1-t)} I_0^2 
     \left(2is \sqrt{t(1-t)} \right) I_1\left( 2 is \sqrt{t(1-t)} \right)
     \diff t  ,
\end{equation*}
since $I_{-1}(x)=I_{1}(x)$ for all $x\in\C$.

\subsection{Approximation in distribution $\boldsymbol{\varrho_{\mathcal{H}}}$ and $\boldsymbol{\varrho_{\mathcal{T}^*}}$}
\label{sec:Main.subsec:direct_simulation}
Consider the  random variable
\begin{align}\label{eq:Yb}
    Y_b
    &:=
    \cos(X_b)+\cos(\beta X_b)+\cos ((1+\beta) X_b),  
\end{align}
where $X_b\sim U[0,b]$, $b>0$, and $\beta>0$ is any irrational number.
Interestingly enough, the random variable $3+2Y_b$ approximates the random eigenvalue $T$ of the triangulation $\mathcal{T}^*$ with density function $f_T$ in distribution as $b\to \infty$. 
Inspired by the Wolfram blog \cite{trott18}, we formulate and prove the following
\begin{theorem}\label{theorem:approximation}
    Let $X_b\sim U([0,b])$ for some $b\in\R_+$. For a random variable $Y_b$ defined in \eqref{eq:Yb},
    it holds
    \begin{align} \label{eq:approxT}
        3+2Y_b 
        \overset{d}{\underset{b\rightarrow\infty}{\longrightarrow}}
        T,
    \end{align}
    \begin{align} \label{eq:approxH}
        A\sqrt{3+2Y_b} 
        \overset{d}{\underset{b\rightarrow\infty}{\longrightarrow}}
        H,
    \end{align}
    where $A$ is a Rademacher distributed random variable independent of $X_b$ and $\overset{d}{\longrightarrow}$ denotes convergence in distribution.
\end{theorem}

To illustrate the above convergence, we set $\beta=\phi:=\frac{1+\sqrt{5}}{2}= 1.61805\ldots$ to be the \textit{golden ratio}.
Since $\phi^2-\phi-1=0$, it follows $1+\beta=\phi^2$.
Hence, we simulate
\begin{align*}
    Y_b
    =
    \cos(X_b) + \cos(\phi X_b) + \cos (\phi^2 X_b),  
\end{align*}
generating $10^5$ realizations of a uniformly distributed random variable $X_b$ on the interval $[0,b]$ for $b=1,10,10^2,10^5$.
Then we compare the empirical density function of $3+2Y_b$ with $f_T$ given in Proposition \ref{proposition:densities_for_sequences}. 
As one can see in Figure \ref{fig:density_phi}, the histograms show convergence to $f_T$ as $b$ goes to infinity. 
\begin{figure}[t]
    \centering
    \begin{minipage}{.5\textwidth}
        \centering
        \includegraphics[width=1\linewidth, height=0.2\textheight]{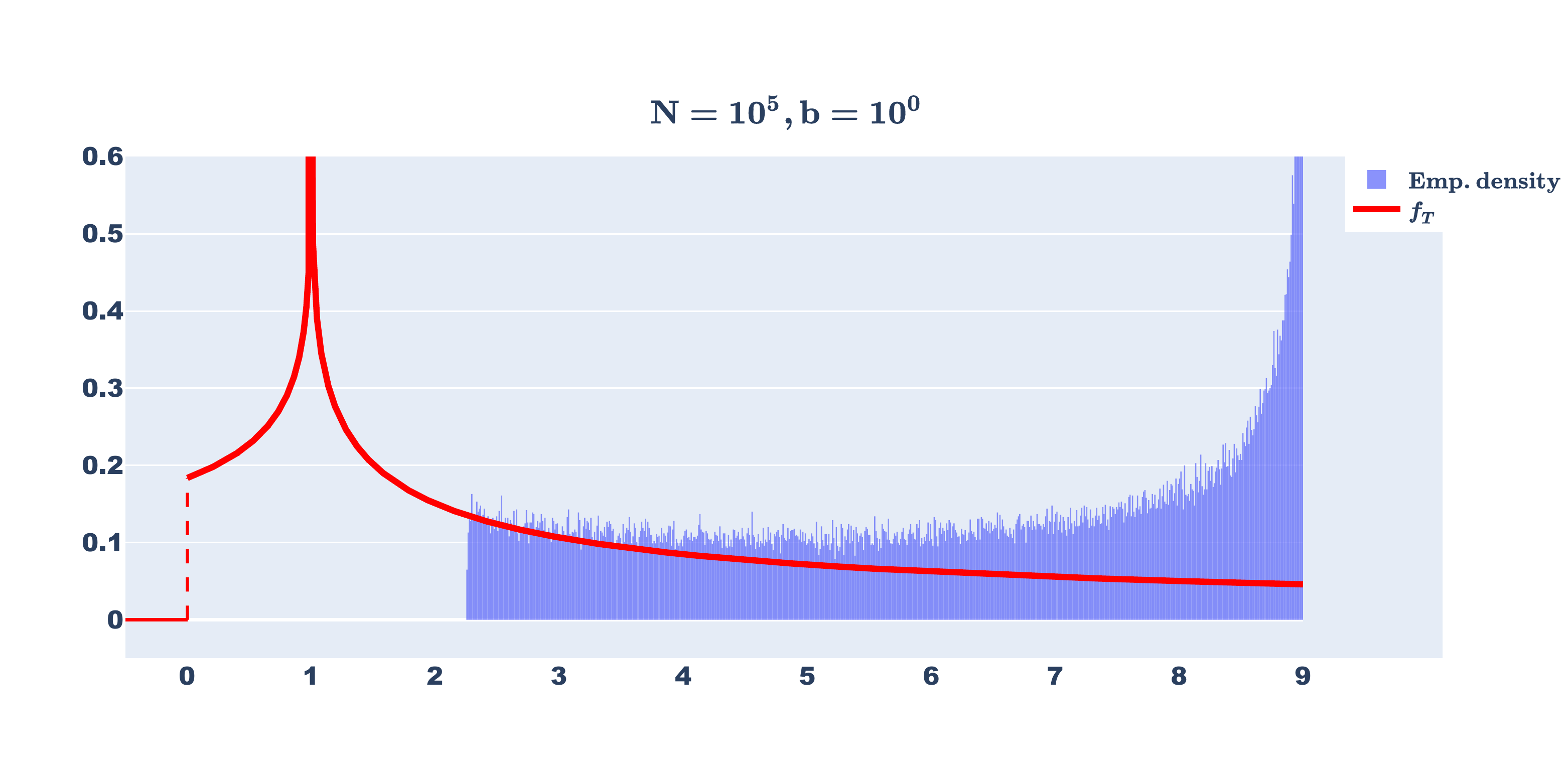}
    \end{minipage}%
    \begin{minipage}{0.5\textwidth}
        \centering
        \includegraphics[width=1\linewidth, height=0.2\textheight]{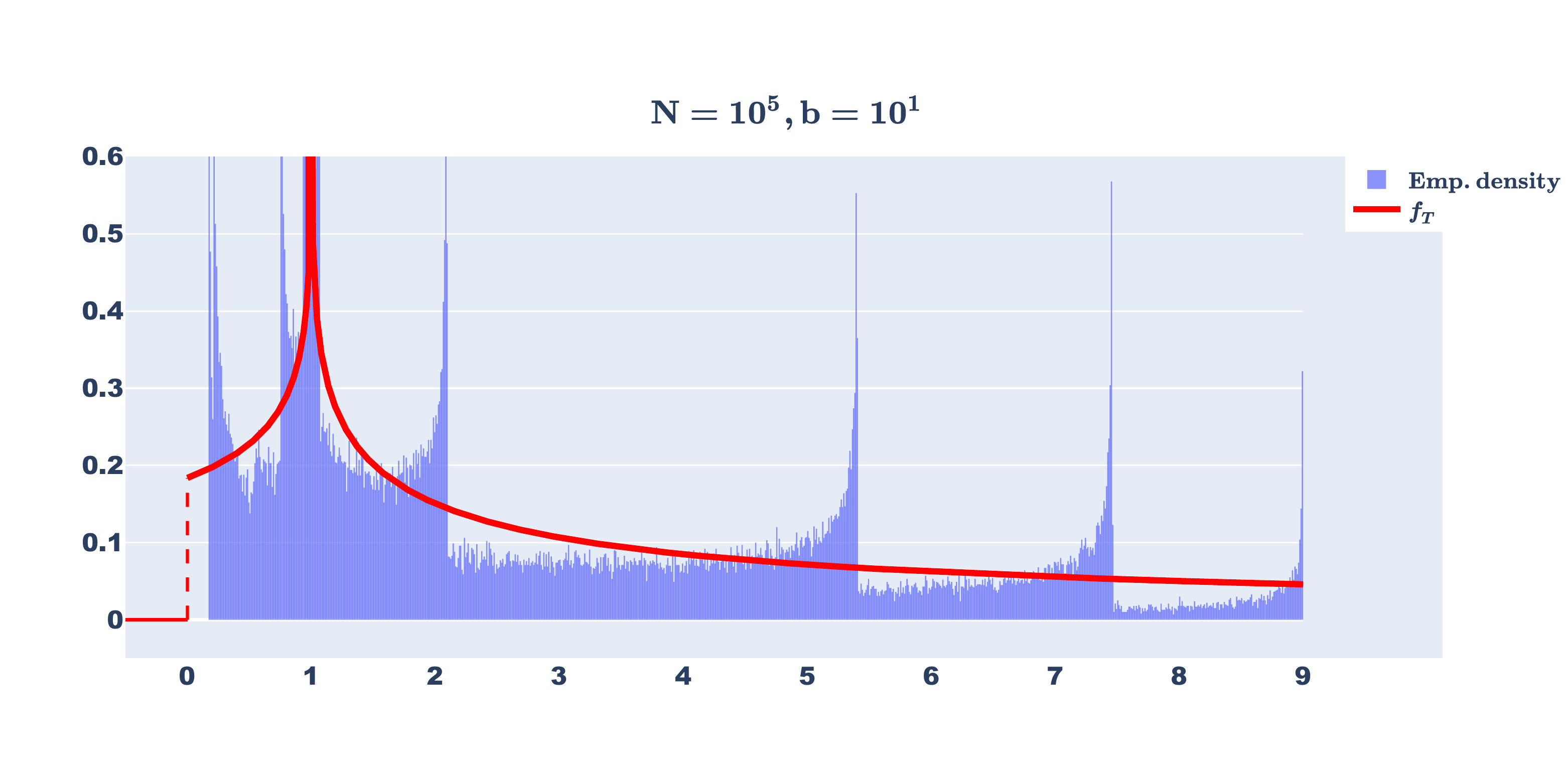}
    \end{minipage}
        \begin{minipage}{.5\textwidth}
        \centering
        \includegraphics[width=1\linewidth, height=0.2\textheight]{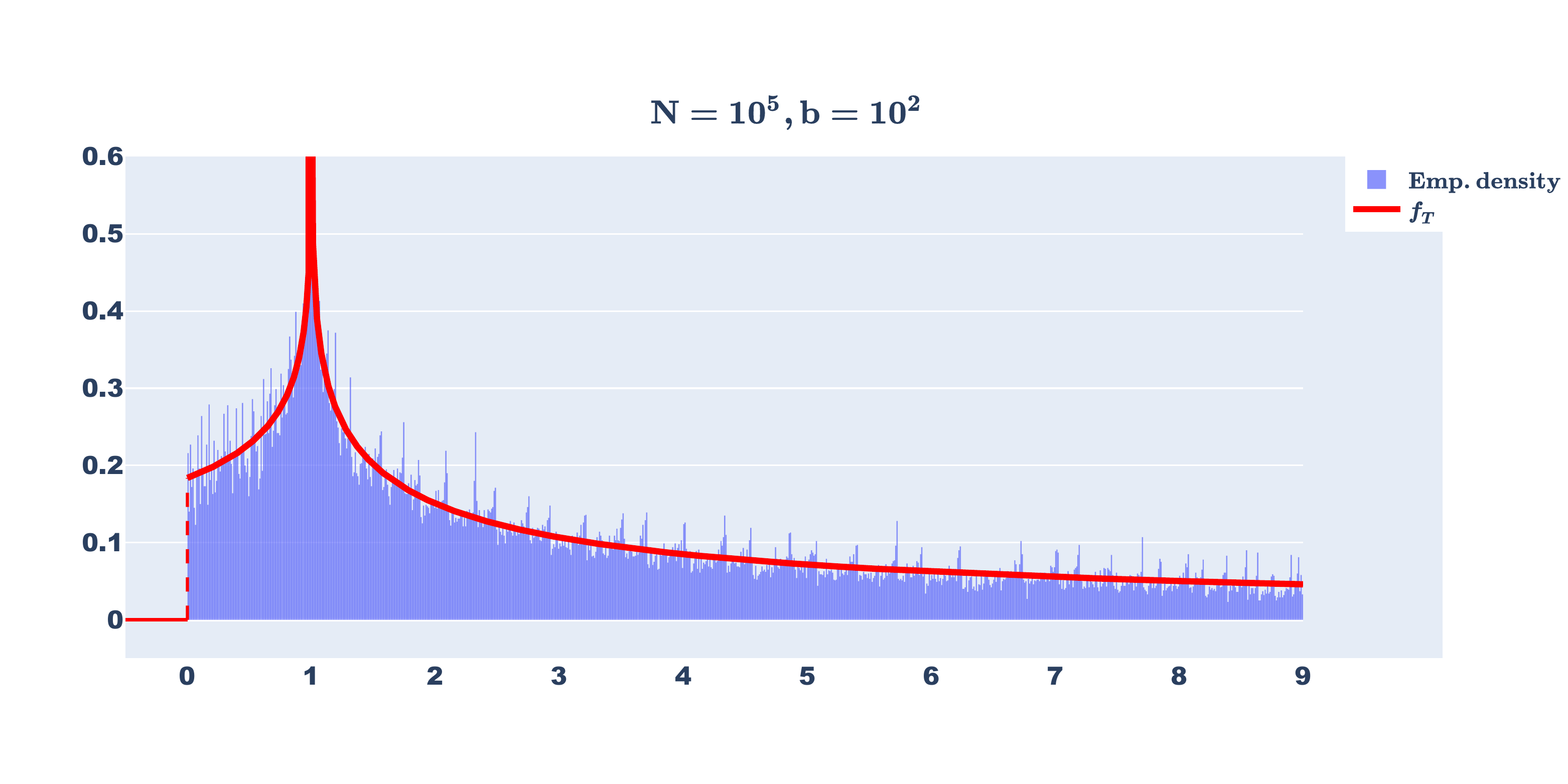}
    \end{minipage}%
    \begin{minipage}{0.5\textwidth}
        \centering
        \includegraphics[width=1\linewidth, height=0.2\textheight]{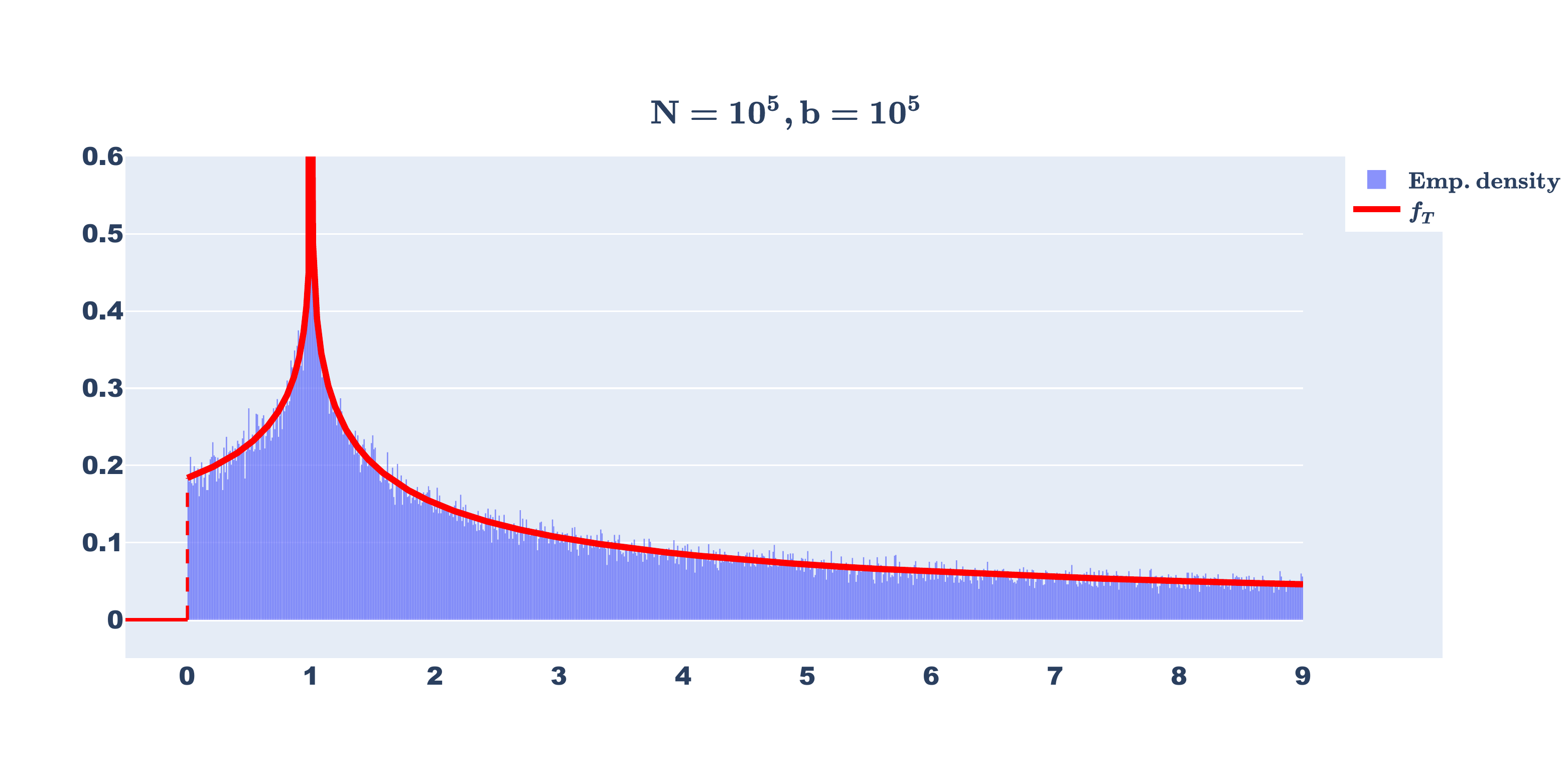}
    \end{minipage}
    \caption{Density $f_T$ (red line) and normalized histograms of simulated $3+2Y_b$ (blue bars) for $b= 1$ (upper left), $b=10$ (upper right), $b=10^2$ (lower left), $b=10^5$ (lower right) with sample size $10^5$ and $\beta=\phi$.} \label{fig:density_phi}
\end{figure}
However, there is an easy direct way of simulation of $T$ and $H$ without any approximations:
\begin{remark}[Exact simulation of $T$ and $H$]
    It holds 
    \begin{align}
        T\eqd 
        3+2\left( \cos (U_1-U_2)+\cos (U_1-U_3)+\cos(U_2-U_3) \right), 
        \label{eq:f_cos}
    \end{align}
    where $U_1,U_2,U_3\sim U[0,2\pi]$ are i.i.d.\ random variables.
    
    Indeed, Theorem \ref{theorem:cotfas_integral} implies that the moments of $T$ can be written as
    \begin{align*}
        \E T^k 
        = 
        \E \left( f(U_1,U_2,U_3)^{k} \right), 
        \quad k\in\N,
    \end{align*}
    where $f(x,y,z):=\left| e^{ix}+e^{iy}+e^{iz}\right|^2$. 
    Since random variables on the left- and right-hand side of \eqref{eq:f_cos} take values in the interval $[0,9]$, their probability laws are uniquely determined by their moment sequences. Hence, it follows $T\eqd f(U_1,U_2,U_3)$.
    Using Euler's identity, write
    \begin{align*}
        f(U_1,U_2,U_3)
        &=
        \left| e^{iU_1}+e^{iU_2}+e^{iU_3}\right|^2
        = 
        3+2\left( \cos (U_1-U_2)+\cos (U_1-U_3)+\cos(U_2-U_3) \right).
    \end{align*}
    As stated in Proposition \ref{prop:3} below, the distribution of $U_3$ can be even chosen arbitrarily, i.e., it need not be uniform.
    Distribution relation \eqref{eq:TdH} allows us to simulate $H$, once $T$ is simulated by means of formula \eqref{eq:f_cos}.
\end{remark}

\begin{remark}
    In 1880, Lord Rayleigh posed in \cite{Rayleigh1880} a problem about the distribution of a finite Fourier series 
    \begin{align*}
        \sum_{j=1}^N a_j \cos(\omega_j t +b_j X_j),
    \end{align*}
    where $X_j$ are random variables and $N, a_j,b_j,\omega_j, t$ are deterministic numbers.
    More than 100 years later, Blevins \cite{Blevins1997} derived a formula for the probability density of such series in case of independent phases $X_j$.
    The above remark shows the distribution of such Fourier series with dependent phases $X_j=U_k-U_l$, $\{j,k,l\}=\{1,2,3\}$, $U_1,U_2,U_3 \sim U[0,2\pi]$ i.i.d.\ and parameters
    \begin{align*}
        N
        &=
        3, 
        \quad a_j=b_j \equiv 1,
        \quad \omega_j \equiv 0, 
        \quad j=1,\ldots,3.
    \end{align*}
\end{remark}

Theorem \ref{theorem:approximation} is proven in Section \ref{sec:proofs.subsec:CharGenFct} by the method of moments using a nontrivial new identity about modified Bessel functions, cf. Theorem \ref{theorem:Besselfunction} below. However, this proof is, in a sense, non-constructive.
In what follows, we give an independent direct and constructive proof of Theorem \ref{theorem:approximation} which reveals connections to ergodic theory. 

Our goal will be to show 
\begin{proposition}
    \label{prop:3}
    Let $\beta>0$ be any irrational number, and let $U_1, U_2\sim {U}[0,2\pi]$ be independent random variables.
    Then
    \begin{align*}
     \cos(X_b) + \cos(\beta X_b) + \cos((1+\beta) X_b) 
     \xrightarrow[b\rightarrow\infty]{d}
     \cos(U_1-U_3) + \cos(U_3-U_2) + \cos(U_2-U_1),
    \end{align*}
    where $U_3$ is an arbitrary random variable independent of $U_1$ and $U_2$.
\end{proposition}
The next main statement bears the spirit of Weyl's uniform distribution theory, cf. \cite{weyl1916,kuipers_niederreiter,granville20}, thus connecting  Proposition \ref{prop:3} to ergodic theory:
\begin{proposition}
    \label{prop:1}
    Let $\beta>0$ be an irrational number.
    Then,
    \begin{equation}\label{weyl}
     (X_b,\beta X_b) ~\mathrm{mod}~ 2\pi 
     \xrightarrow[b\rightarrow\infty]{d} (U_1,U_2),
    \end{equation}
    where
    $U_1, U_2\sim {U}[0,2\pi]$ are independent random variables. 
\end{proposition}
Its proof via the method of moments can be found in Section \ref{sec:Proofs.subsec:ApproxEigenH}.
By Proposition \ref{prop:1} and the continuous mapping theorem, we arrive at
\begin{align*}
 (X_b,\beta X_b,(1+\beta) X_b) ~\mathrm{mod}~ 2\pi 
 \xrightarrow[b\rightarrow\infty]{d}
 (U_1,U_2, U_1+U_2) ~\mathrm{mod}~ 2\pi .
\end{align*}
The conditional distribution below (when $U_3$ is given) equals
\begin{equation}\label{conditional}
    \left( (U_1-U_3,U_2+U_3) ~\mathrm{mod}~ 2\pi \right) \; | U_3 \; \eqd \; (U_1,U_2).
\end{equation}
Since the right-hand side of (\ref{conditional}) is independent of $U_3$, 
identity (\ref{conditional}) holds also unconditionally. Then
\begin{align*}
    (U_1,U_2, U_1+U_2) ~\mathrm{mod}~ 2\pi
    & \; \eqd \; (U_1-U_3,U_2+U_3, (U_1-U_3)+(U_2+U_3)) ~\mathrm{mod}~ 2\pi \\
    & \; = (U_1-U_3,U_2+U_3, U_1+U_2) ~\mathrm{mod}~ 2\pi \\
    & \; \eqd \; (U_1-U_3, 2 \pi -U_2+U_3, U_1+2\pi - U_2) ~\mathrm{mod}~ 2\pi \\
    & \; \eqd \; (U_1-U_3, -U_2+U_3, U_1 - U_2) ~\mathrm{mod}~ 2\pi,
\end{align*}
where we used the obvious relation $U_2 \eqd 2\pi- U_2$. Using the continuous mapping theorem once again, we finally get
\begin{align*}
    \cos(X_b)+\cos(\beta X_b)+\cos((1+\beta) X_b)
    &\xrightarrow[b\rightarrow\infty]{d}
    \cos(U_1)+\cos(U_2)+\cos(U_1+U_2) \\
    & \eqd 
    \cos(U_1-U_3)+\cos(U_3-U_2)+\cos(U_2-U_1),
\end{align*}
which finishes the constructive proof of Proposition \ref{prop:3}.

\subsection{Auxiliary results for modified Bessel functions}
\label{sec:Main.subsec:AuxBessel}
Results in this section are instrumental to prove Proposition \ref{prop:charfctHT} and Theorem \ref{theorem:approximation}. However, they may be also of independent interest (especially Theorem \ref{theorem:Besselfunction} below) since they state some fundamental identities about modified Bessel functions.

For convenience, introduce the notation $a_k:=\mu_k\left(\mathcal{T}^*\right)$, $k\in\N_0$, where the sequence $\mu\left(\mathcal{T}^*\right)$ was given in \eqref{eq:mu_T*}.

Using Vandermonde's identity \cite{abra}, rewrite $a_k$ for every $k\in \N_0$ in the following way: 
\begin{align}
    a_k 
    &= 
    \sum_{k_1=0}^k\sum_{k_2=0}^{k-k_1}\binom{k}{k_1,k_2,k-k_1-k_2}^2 
    = 
    \sum_{k_1=0}^k\frac{(k!)^2}{(k_1!)^2}\sum_{k_2=0}^{k-k_1}\frac{1}{(k_2!)^2((k-k_1-k_2)!)^2}\cdot \frac{((k-k_1)!)^2}{((k-k_1)!)^2} \notag\\
    &=  
    \sum_{k_1=0}^k \frac{(k!)^2}{(k_1!)^2((k-k_1)!)^2}\sum_{k_2=0}^{k-k_1} \binom{k-k_1}{k_2}^2 
    = 
    \sum_{k_1=0}^k \binom{k}{k-k_1}^2 \binom{2(k-k_1)}{k-k_1} 
    = 
    \sum_{n=0}^k\binom{k}{n}^2\binom{2n}{n}. \label{a_k}
\end{align}

The sequence $a=\{a_k\}_{k\in\N_0}$ has number A002893 in \textit{the On-Line Encyclopedia of Integer Sequences} \cite{Sloane10}. 
The following result was communicated to us by Vladeta Jovovic:  
\begin{proposition}\label{prop:vladetta}
It holds
\begin{align}
\sum_{k=0}^\infty a_k \frac{x^k}{(k!)^2} &= \left(\sum_{k=0}^\infty \frac{x^k}{(k!)^2} \right)^3=I_0^3\left(2\sqrt{x}\right),\quad x\in\C. \label{formula:vladetta}
\end{align}
\end{proposition}
Our independent proof will be given in Section \ref{sec:proofs.subsec:ModBessel}. This implies the following form of exponential generating functions of the sequence $a$ and its subsequences containing only terms of even or odd order: 
\begin{corollary}\label{prop:generating function}
    For arbitrary $x\in\C$, it holds
    \begin{align}
        \sum_{k=0}^\infty a_{k} \frac{x^{k}}{k!} 
        &=
        \int_0^{\infty} I_0^3\left(2\sqrt{xt}\right) e^{-t}\diff t, 
        \label{eq:generating_function_T}\\
        \sum_{k=0}^\infty a_{2k} \frac{x^{2k}}{(2k)!} 
        &= 
        \frac{1}{2}\int_0^{\infty} \left( I_0^3\left(2\sqrt{xt}\right) + I_0^3\left(2i\sqrt{xt}\right) \right) e^{-t}\diff t, 
        \label{eq:generating_function}\\
        \sum_{k=0}^\infty a_{2k+1} \frac{x^{2k+1}}{(2k+1)!} 
        &= 
        \frac{1}{2}\int_0^{\infty} \left( I_0^3\left(2\sqrt{xt}\right) -I_0^3\left(2i\sqrt{xt}\right) \right)e^{-t}\diff t. \label{eq:generating_function_T_odd}
    \end{align}
\end{corollary}
The following beautiful identity finishes the line of our results:
\begin{theorem}\label{theorem:Besselfunction}
    For every $x\in\C$, it holds
    \begin{equation}
        e^{3x/2} \sum_{n\in\Z}I_n^3 (x)
        = 
        \int_0^\infty  I_0^3\left(\sqrt{2x t}\right) e^{-t} \diff  t.  
        \label{identity1} 
    \end{equation}	
\end{theorem}

\section{Local weak convergence of graphene and fullerenes}\label{sec:conjecture}

In this section, we explain the difficulties arising in the proof of Conjecture \ref{conjecture}.

\subsection{Working with duals: triangulations with degree constraints}

It is easily seen that the local weak convergence of a sequence of 3--connected planar graphs is equivalent to the local weak convergence of the sequence of its duals; 
we recall that the dual $G'$ of a planar graph $G$ has nodes representing the facets of $G$, with two nodes being connected by an edge in $G'$ if and only if their facets in $G$ have a common edge. 
The following paragraphs will only deal with the local weak convergence of the duals of fullerenes, which are triangulations with degree constraints. 

Every fullerene $F_n$ has exactly $12$ pentagonal faces (and $m=\sfrac{n}{2}-10$ hexagonal faces), so its dual $T_m$ has $12$ nodes with degree $5$ and $m$ with degree $6$; 
moreover, each face of $T_m$ is a triangle, since fullerenes are $3$-regular. 
In summary, the set of dual graphs of fullerenes is precisely the set of sphere triangulations with $12$ nodes of degree $5$ and the rest of degree $6$; 
they must also be 3-connected. 
We simply call these triangulations \emph{f-triangulations}. 

\begin{remark}
    Unlike sphere triangulations with no constraints on the degree, the number $F(m)$ of f-triangulations with $m$ vertices of degree 6 is not amenable to classical analytical methods like the ones worked out by Tutte. 
    Indeed, in the celebrated paper \cite{Thurs98}, William Thurston proved that $F(m)$ has order $m^9$ using deep geometric arguments, but no explicit formula is known. 
    The best bounds for this limit known to the authors are presented in \cite{Rukh18}, which are direct consequences of \cite{engel2018number}: 
    \begin{align*}
        \liminf_{m\rightarrow \infty} \frac{F(m)}{m^9} 
        = 
        \frac{809}{2^{15}\cdot 3^{13}\cdot 5^2},\quad
        \limsup_{m\rightarrow\infty} \frac{F(m)}{m^9} 
        =
        \frac{809\zeta(9)}{2^{15}\cdot 3^{13}\cdot 5^2},
    \end{align*}
    where $\zeta(9)= 1.002008\ldots$ is the value of the Riemann zeta function.
\end{remark}

\subsection{The weak limit of triangulations}
We now fix a sequence $\lbrace T_m\rbrace$ of random f-triangulations. 

Since the set of graphs with uniformly bounded degree is a precompact subset in the set of rooted graphs endowed with the local distance, see \cite{benjaminicurien12}, the sequence $\lbrace T_m\rbrace$ has weak limits. 
Upon extracting a subsequence, we can suppose it has a local weak limit, which is a random rooted graph $(G,o)$. 
This infinite graph should have the following properties: 
(i) the root should have degree $6$ or $5$; 
(ii) $(G,o)$ should be unimodular, as every local weak limit of graphs. 

It turns out that any weak limit of \emph{triangulations} has a mean degree of exactly 6 at the root; more precisely, $\E[\mathrm{deg}_G(o)]=6$, see \cite{albenque2021local}. 
Consequently, if $(G,o)$ is a weak limit of a sequence of triangulations, the degree of the root is almost surely \emph{equal} to 6. 
By a standard argument using unimodularity (\!\!\cite{Aldous2007}), the degree of \emph{all} vertices in $G$ must be 6. 
In other words, any weak limit of f-triangulations must be an infinite $6$-regular triangulation. 
One example is the dual of the planar hexagonal lattice $\mathcal{H}$. 
However, $\mathcal{H}$ is not the unique infinite $6$-regular planar triangulation: any \emph{folding} of $\mathcal{H}$ would have the same property. 
Yet, there is only one unimodular $6$-regular infinite triangulation of the plane \emph{with only one topological end}, and it is $\mathcal{H}$; hence, proving Conjecture \ref{conjecture} is reduced to proving that no weak limit of f-triangulations can have more than one topological end. 

\begin{remark}\label{rk:nanotubes}
    \begin{enumerate}[(i)]
        \item  It is possible to construct a sequence of f-triangulations 
            that converges towards some folding of $\mathcal{H}$; in chemistry, these structures are known as \emph{carbon nanotubes}. Conjecture \ref{conjecture} implies that these structures are rare in regard to the whole set of f-triangulations. 
        \item  Most generating algorithms, like \textit{Buckygen} 
            \cite{buckygen}, which is the most efficient known complete algorithm, or the approach of Buchstaber and Erokhovets in \cite{BuchEro17}, use a hierarchical trial-and-error approach. 
            By applying carefully chosen finite sequences of \textit{expansion} and \textit{reduction operations} to the dodecahedron (the smallest fullerene) and $C_{28}$, ultimately generates the entire set of fullerenes for a given number of vertices. 
            However, not every sequence of operations leads to a (new) fullerene. 
            Therefore, the practical construction of a uniformly distributed sample among all fullerenes with $n$ vertices is a non--trivial problem.
    \end{enumerate}
\end{remark}

\subsection{Isoperimetric properties of triangulations}

The proof of the one-endedness of weak limit of planar graphs usually proceeds by contradiction, as demonstrated in the seminal argument of Corollary 3.4 in \cite{angel2003uniform}. 

Let us denote $(T,o)$ as the weak limit of $\lbrace T_m\rbrace$, and suppose that the probability of $(T,o)$ having two distinct topological ends is $\varepsilon >0$.
This implies that, for some fixed integer $l$, with probability $\varepsilon$, there is a closed path of finite length $l$ such that $T$ with this path removed has two unbounded connected components. 
According to \eqref{def:BSConv}, this leads to the following statement: there exists a sequence $\lbrace c_j\rbrace_{j\in\N_0}$ with $c_j\rightarrow\infty$ as $j\rightarrow\infty$, and a fixed integer $l$ such that 
\begin{equation}\label{target} 
    \frac{\left| E(m,l,c_j)\right|}{\left|\lbrace\text{f-triangulations of size } m\rbrace\right|} \xrightarrow[n\rightarrow\infty]{}0,
\end{equation}
where $\left| A \right|$ denotes the cardinality of a set $A$, and $E(m,l,c)$ represents the set of f-triangulations with a closed path of length $l$ and such that $T$ deprived of this path has two connected components of size greater than $c$. One can think of $E(m,l,c)$ as graphs with a bottleneck. 

\begin{remark}
    The isoperimetric inequality from \cite{angel2018isoperimetric} implies that if a fullerene has a bottleneck path of length $l$ and $c_j > l^2/12$, then both connected components of the fullerene deprived of the path must have at least one pentagon. It is worth mentioning that in general, the isoperimetric inequality mentioned above \emph{fails} for triangulations allowed to have degrees strictly smaller than 5 (In Remark \ref{rk:nanotubes} we mention the existence of f-triangulations with an arbitrary small isoperimetric constant). 
\end{remark} 

To prove \eqref{target}, a classical argument is needed, similar to the one in \cite{angel2003uniform}. 
Let us denote by $A(m,l,p)$ the set of all planar graphs with $m$ vertices, in which all faces are triangles except one which must have $l$ faces, and in which all the internal nodes have degree $6$ except exactly $p$ having degree $5$; then, 
\begin{align*}
    E(m,l,c)\leq 
    \sum_{p_1 + p_2 \leq 12}
    \sum_{k=c}^{m-c} 
    A(k,l, p_1)A(m-k,l,p_2).
\end{align*}

For triangulations without degree constraints, estimates for the number of triangulations with a non-triangular face were available thanks to Tutte's works; but, this is not the case when degree constraints are added. In conclusion, no efficient bound for $A(k,l,p)$ is known, and we could not extract it from Thurston's paper.

\section{Proofs}\label{sec:proofs}
We start proving statements about Bessel functions formulated in Section \ref{sec:Main.subsec:AuxBessel}.

\subsection{Results for modified Bessel functions}\label{sec:proofs.subsec:ModBessel}

Proof of Proposition \ref{prop:vladetta}:
\begin{proof}
    By the convolution formula for two infinite series, 
    it holds for any two complex sequences 
    $\lbrace a_k\rbrace_{k\in\N_0}$, $\lbrace b_k\rbrace_{k\in\N_0}$ 
    that
    \begin{align*}
        \left( \sum_{k=0}^\infty a_k x^k \right)\cdot 
        \left( \sum_{k=0}^\infty b_k x^k \right) 
        = 
        \sum_{k=0}^\infty c_k x^k,
    \end{align*}
    where 
    \begin{align*}
        c_k: = \sum_{n=0}^k a_{k-n} b_n.
    \end{align*}
    Hence, it holds
    \begin{align*}
        \left(\sum_{k=0}^\infty \frac{x^k}{(k!)^2} \right)^2 
        &= 
        \left(\sum_{k=0}^\infty \frac{1}{(k!)^2} x^k \right) \cdot \left(\sum_{k=0}^\infty \frac{1}{(k!)^2} x^k \right) 
        = 
        \sum_{k=0}^\infty c_k x^k 
        = 
        \sum_{k=0}^\infty \frac{1}{(k!)^2}\binom{2k}{k} x^k,
    \end{align*}
    since
    \begin{align*}
        c_k 
        &= 
        \sum_{n=0}^k \frac{1}{((k-n)!)^2}\cdot \frac{1}{(n!)^2} \cdot \frac{(k!)^2}{(k!)^2} 
        = 
        \frac{1}{(k!)^2}\sum_{n=0}^k \binom{k}{n}^2 
        = 
        \frac{1}{(k!)^2}\binom{2k}{k}.
    \end{align*}
    We can compute the middle term in (\ref{formula:vladetta}) analogously:
    \begin{align*}
        \left(\sum_{k=0}^\infty \frac{x^k}{(k!)^2} \right)^3 
        &= 
        \left(\sum_{k=0}^\infty \frac{x^k}{(k!)^2} \right)^2\cdot \left(\sum_{k=0}^\infty \frac{x^k}{(k!)^2} \right) 
        = 
        \left( \sum_{k=0}^\infty \frac{x^k}{(k!)^2}\binom{2k}{k} \right)\cdot\left(\sum_{k=0}^\infty \frac{x^k}{(k!)^2} \right) 
        = 
        \sum_{k=0}^\infty d_k x^k 
        = 
        \sum_{k=0}^\infty a_k \frac{x^k}{(k!)^2},
    \end{align*}
    where 
    \begin{align*}
        d_k 
        = 
        \sum_{n=0}^k 
        \frac{1}{((k-n)!)^2}\binom{2(k-n)}{k-n}\frac{1}{(n!)^2}\cdot \frac{(k!)^2}{(k!)^2} 
        = 
        \frac{1}{(k!)^2} \sum_{n=0}^k \binom{k}{k-n}^2\binom{2(k-n)}{k-n} 
        = 
        \frac{1}{(k!)^2} a_k.
    \end{align*}
    The last equality holds due to (\ref{a_k}).
    The second equation of (\ref{formula:vladetta}) holds due to the definition of the modified Bessel function of first order.
\end{proof}

Proof of Corollary \ref{prop:generating function}:

\begin{proof}
    Using formula (\ref{formula:vladetta}) and the identity
    \begin{align*}
        \int_0^\infty \frac{t^{k}}{k!}e^{-t}\diff t =1, \quad k\in\N_0
    \end{align*}
    yields
    \begin{align*}
        \sum_{k=0}^\infty \frac{x^k}{k!} a_k 
        =   
        \int_0^\infty \sum_{k=0}^\infty \frac{(xt)^k}{(k!)^2} a_k e^{-t} 
        \diff t 
        = 
        \int_0^\infty I_0^3\left(2 \sqrt{xt} \right) e^{-t}
        \diff t ,
    \end{align*}
    where we interchanged the order of the sum and the integral above. 
    This proves the relation \eqref{eq:generating_function_T}.
    
    To show \eqref{eq:generating_function} and \eqref{eq:generating_function_T_odd}, notice that 
    \begin{align*}
        \int_0^\infty I_0^3\left(2 \sqrt{xt} \right) e^{-t}\diff t 
        &= 
        \sum_{k \text{ even} }^\infty \frac{x^k}{k!} a_k + \sum_{k \text{ odd}}^\infty \frac{x^k}{k!} a_k, \\ 
        \int_0^\infty I_0^3\left(2 \sqrt{-xt} \right) e^{-t}\diff t 
        &= 
        \sum_{k \text{ even}}^\infty \frac{x^k}{k!} a_k 
        - 
        \sum_{k \text{ odd}}^\infty \frac{x^k}{k!} a_k. 
    \end{align*}
    Summing up the left- and the right-hand sides of the above relations and dividing by two yields
    \begin{align*} 
        \sum_{k \text{ even}}^\infty \frac{x^k}{k!} a_k 
        &= 
        \frac{1}{2}\int_0^\infty \left( I_0^3\left(2 \sqrt{xt} \right) 
        +
        I_0^3\left(2 i\sqrt{xt} \right) \right) e^{-t}\diff t,\\
        \sum_{k \text{ odd}}^\infty \frac{x^k}{k!} a_k 
        &= 
        \frac{1}{2}\int_0^\infty \left( I_0^3\left(2 \sqrt{xt} \right) 
        -
        I_0^3\left(2 i\sqrt{xt} \right) \right) e^{-t}\diff t,
    \end{align*}
    which finishes the proof.
\end{proof}

Proof of Theorem \ref{theorem:Besselfunction}:
\begin{proof}
    Rewrite the statement (\ref{identity1}) in a more convenient way by taking $e^{3x/2}$ to the right-hand side and evaluating the whole expression at $2x$: 
    \begin{equation}
        \sum_{n\in\Z}I_n^3 (2x)
        = 
        \int_0^\infty I_0^3\left(2\sqrt{x t}\right) e^{-t-3x} \diff t.
    \label{identity1_mod} 
    \end{equation}
    
    The main idea of this proof is to show that both sides of (\ref{identity1_mod}) are the unique solution of the initial value problem
    \begin{align*}
        A_x u(x) = 0,
    \end{align*}
    subject to the initial conditions:
    \begin{align*}
        u(0) = 1,\quad 
        \left. D_x u(x)\right|_{x=0} = 0,\quad 
        \left.D_x^2 u(x)\right|_{x=0} = 6.
    \end{align*}
    To obtain the differential operator $A_x$,
    the Wolfram Mathematica package \texttt{HolonomicFunction}, V.~1.7.3 \cite{koutchan2009,koutchan2010}
    is available.
    
    \textbf{Step 1.} We obtain the annihilator for the right-hand side of (\ref{identity1_mod}). 
    The command
    \begin{equation*}\tt
        ann 
        = 
        Annihilator[
          Integrate[Exp[-3x-t] 
            BesselI[0, 2\,Sqrt[x t]]^3, \{t,0,Infinity\}], 
          Der[x]]
    \end{equation*}
    yields the annihilator
    \begin{align}\label{eq:Annihil}
        A_x 
        = 
        x^2 D_x^3 - \left(x^2-3 x\right) D_x^2 - 
        \left(24 x^2+2 x-1\right) D_x - \left(36 x^2+24 x\right),
    \end{align}
    where $D_x=\partial/\partial x$ represents the differential operator with respect to $x$.
    
    \textbf{Step 2.} We first see that the left-hand side of (\ref{identity1_mod}),
    \begin{equation}
    \label{convergence0}
     b(x) = \sum_{n\in\Z}I_n^3(2x),
    \end{equation}
    exists.
    Since for $n\in\Z$,
    \[
     I_n(2x) = \sum_{m=0}^\infty \frac{x^{n+2m}}{m!\Gamma(n+m+1)},
    \qquad I_{-n}(2x) =I_n(2x), \qquad |I_n(2x)|\le I_n(2|x|),
    \]
    we have for $n\ge 0$,
    \begin{equation}\label{bound}
        |I_n(2x)| 
        \le 
        \frac{|x|^n}{n!} \sum_{m=0}^\infty \frac{|x|^{2m}}{m!} 
        = 
        \frac{|x|^n}{n!} e^{|x|^2},
        \qquad
        |I^3_n(2x)| 
        \le 
        \frac{|x|^{3n}}{n!} e^{3|x|^2}
    \end{equation}
    and hence (\ref{identity1_mod}) is bounded above by
    $|b(x)|\le 2 e^{|x|^3+3|x|^2}$, $x\in\C$.
    
    To show that $A_x$ is the annihilator for $b(x)$, we use the telescoping method.
    The command
    \[
    \tt
     \{ann, Qn\} = Simplify[CreativeTelescoping[BesselI[n, 2 x]^3, S[n] - 1, Der[x]]]
    \]
    yields
    \[
     A_x I_n(2x)^3 - g_{n+1}(x) + g_n(x) = 0,
    \]
    where $A_x$ is the same operator given in (\ref{eq:Annihil}) and
    \[
     g_n(x) = Q_{x,n} I_n^3(2x)
    \]
    with
    \begin{align*}
        Q_{x,n} 
        =& 
        \frac{1}{8} D_x \left(4 n^2+n (25-12 x)+8 x (2 x-1)\right)+\frac{1}{8} D_x^2 (x (25 n-4 x)) \\
        & 
        -\frac{3 \left(7 n^3+4 n x (21 x+1)+48 x^3\right)}{8 x} + 
        \left(24 x^2\right) S_n,
    \end{align*}
    where $S_n$ is the shift operator $S_n f(n)=f(n+1)$.
    Applying the three-term relation 
    \[
     n I_n(2x) = x (I_{n-1}(2x)-I_{n+1}(2x)),
    \]
    $g_n(x)$ is simplified as
    \begin{align*}
         g_n(x) 
         = 
         3 x n^{-3} I_{n-1}(2x)\Bigl(
         & 8 n^3 x I_{n-1}^2(2 x) \\
         & + n (n^3 - 4 n^2 x + 16 x^3) I_{n-1}(2 x) I_n(2 x) \\
         & + n^2 (n^2-24 x^2+n (4x-1)) I_{n}^2(2x) \Bigr).
    \end{align*}
    By (\ref{bound}), we have $\lim_{n\to\infty} g_{n}(x)=0$, and hence
    \begin{equation}\label{convergence}
    \sum_{n=-\infty}^\infty A_x I_n^3(2x)
         = \lim_{n,n'\to\infty}\sum_{m=-n}^{n'} A_x I_m^3(2x)
         = \lim_{n,n'\to\infty} (-g_{n'+1}(x)+g_{-n}(x)) 
         = 0
    \end{equation}
    for each $x\in\C$.
    
    Next, we will see that
    \begin{equation}\label{A-changeable}
         A_x \sum_{n=-\infty}^\infty I_n^3(2x)
         = 
         \sum_{n=-\infty}^\infty A_x I_n^3(2x).
    \end{equation}
    Recall that each term of $A_x$ is of the form $x^j D_x^k$, $D_x=\partial/\partial x$.
    It suffices to show that
    \begin{equation}\label{D-changeable}
         D_x^k \sum_{n=-\infty}^\infty I_n^3(2x)
         = 
         \sum_{n=-\infty}^\infty D_x^k I_n^3(2x).
    \end{equation}
    For $h\ne 0$, let
    \[
     \Delta_x^h f = (f(x+h)-f(x))/h.
    \]
    Then, by checking the absolute convergence, it holds
    \[
     (\Delta_x^h)^k \sum_{n=-\infty}^\infty I_n^3(2x)
     = 
     \sum_{n=-\infty}^\infty (\Delta_x^h)^k I_n^3(2x),
    \]
and because there exists $F_n^k(x)$ such that\footnote{
    For example, when $k=1$,
    $|\Delta_x^h I_n^3(2x)|
    =
    6 |I_n(2(x+\theta h))|^2|\dot I_n(2(x+\theta h))|
     \le 6 |I_n(2(|x|+1))|^2||\dot I_n(2(|x|+1))| =:F^1_n(x)$, and $\sum_{n=-\infty}^\infty F^1_n(x)<\infty$ by (\ref{eq:derivative}) and (\ref{bound}).}
    \[
     \sup_{|h|<\varepsilon}|(\Delta_x^h)^k I_n^3(2x)| \le F_n^k(x), \quad \sum_{n=-\infty}^\infty F^k_n(x)<\infty,
    \]
    by the dominated convergence we have 
    \[
     \lim_{h\to 0}(\Delta_x^h)^k \sum_{n=-\infty}^\infty I_n^3(2x)
     = 
     \sum_{n=-\infty}^\infty \lim_{h\to 0}(\Delta_x^h)^k I_n^3(2x).
    \]
    That is, (\ref{D-changeable}), and hence (\ref{A-changeable}) holds.
    Now, by combining (\ref{A-changeable}) and (\ref{convergence}), we establish
    \[
     A_x \sum_{n=-\infty}^\infty I_n^3(2x)
= 0, \quad x\in\C.
    \]
 
    \textbf{Step 3.} Here we note that the left- and right-hand sides of (\ref{identity1_mod}) are entire functions.
    Let $\Omega\ni 0$ be a bounded region, e.g., $\Omega=\{x\in\C\,|\,|x|\le 1\}$.
    The convergence in (\ref{convergence0}) is uniform on $\Omega$, the limit $b(x)=\sum_{n=-\infty}^\infty I_n(2x)^3$ is analytic on $\Omega$ as well
    (Weierstrass's theorem, \cite{ahlfors78}).
    Moreover, the limit function $b(x)$ has an annihilator $A_x$ whose possible singularities are $x=0$ and $\infty$ (since the coefficient of the derivative of the highest degree is $x^2$).
    Hence, $b(x)$ is an entire function whose singularity is $x=\infty$ only.
    The right-hand side of (\ref{identity1_mod}) is also an entire function. 
    
    \textbf{Step 4.}
    We show that both sides of (\ref{identity1_mod}) coincide at least as a point up to the second derivatives.
    If $A_x$ is the annihilator for the analytic function
    \begin{align*}
     g(x) := \sum_{j=0}^\infty c_j x^j,
    \end{align*}
    the coefficients have to satisfy the recurrence relation
    \begin{align*}
    c_j 
    = 
    \frac{1}{j^3}
    \Bigl( (j-1)j c_{j-1} + 24(j-1) c_{j-2} + 36 c_{j-3}\Bigr),
    \quad j\in\N_0.
    \end{align*}
    Therefore, once $g(0)=c_0$, $g'(0)=c_1$, $g''(0)=2 c_2$ are given, the whole sequence $\{c_j\}_{j\in\N_0}$ is determined uniquely.
    To prove (\ref{identity1_mod}), it remains to show that the initial values of its left- and right-hand sides as well as their first two derivatives at $x=0$ coincide.
    
    By $I_n(0)=\1(n=0)$ and the derivative formula (\ref{eq:derivative}), one gets
    \begin{align*}
        \sum_{n\in\Z}I_n^3(2x) \Big|_{x=0} 
        =& 
        \sum_{n\in\Z}I_n^3(0) = 1, \\
        D_x \sum_{n\in\Z}I_n^3(2x) \Big|_{x=0} 
        =& 
        \sum_{n\in\Z} 6 I_n^2(0) I_{n+1}(0) = 0, \\
        D_x^2 \sum_{n\in\Z} I_n^3(2x) \Big|_{x=0} 
        =&
        \sum_{n\in\Z} \Bigl(12 I_{n-1}(0) I_n^2(0) I_{n+1}(0)
        + 12 I_{n}(0) I_{n+1}^2(0) + 6 I_{n}^3(0) + 6 I_{n}^2(0) I_{n+2}(0)\Bigr) = 6.
    \end{align*}
    On the other hand, it is easy to see that
    \begin{align*}
        f(t;0) = e^{-t}, \qquad
        D_x f(t;0) = 3 e^{-t}(t-1), \qquad 
        D_x^2 f(t;0) = \frac{3}{2} e^{-t} (5 t^2-12 t+6),
    \end{align*}
    and, hence
    \begin{align*}
        \int_0^\infty f(t;0) = 1, \qquad
        \int_0^\infty D_x f(t;0) = 0, \qquad
        \int_0^\infty D_x^2 f(t;0) = 6,
    \end{align*}
    which proves the claim.
\end{proof}

\subsection{Densities and characteristic functions of $\boldsymbol{H}$ and $\boldsymbol{T}$} \label{sec:proofs.subsec:CharGenFct}
Proof of Proposition \ref{proposition:densities_for_sequences}:
\begin{proof}
    For $k\in\N$, Proposition \ref{prop:moments_X} and Proposition \ref{prop:number of closed walks} imply that 
    \begin{align}
        \mu_{2k}(\mathcal{H})=\int_{\R} x^{2k}f_X(x)\diff x. \label{proof_formula1}
    \end{align}
    Next, we need to adjust $f_X$ such that (\ref{proof_formula1}) also holds for odd orders $2k+1$, $k\in\N_0$.
    Since $\mu_{2k+1}(\mathcal{H})=0$, the corresponding density has to be symmetric. 
    Therefore, we expand the support of $f_X$ onto the whole real line by replacing $x$ by its absolute value. 
    Further, we normalize the density dividing it by 2. 
    It follows immediately that $f_H$ is a density, and its $k^{\text{th}}$ moment coincides with $\mu_{k}(\mathcal{H})$ for $k\in\N$.
    Due to Proposition \ref{pro:density_X_support}, the support of $f_H$ is the interval $[-3,-3]$. Thus, the distribution of $H$ is uniquely identified by its moment sequence with density $f_H$, cf.\ e.g., \cite{Kallen02}. 
    
    Using the relationship $\mu_k\left(\mathcal{T}^*\right)=\mu_{2k}(\mathcal{H})$, $k\in\N_0$, and changing variables $x \mapsto \sqrt{z}$ we get
    \begin{align*}
        \mu_k\left(\mathcal{T}^*\right)&=\int_0^\infty x^{2k} f_H(x)\diff x=\int_0^\infty x^{2k} \int_0^\infty txJ_0(tx)J_0^3(t)\diff t \diff x\\
        &=\int_0^\infty \frac{z^k}{2\sqrt{z}} \int_0^\infty t \sqrt{z} J_0(t\sqrt{z})J_0^3(t) \diff t \diff z 
        = \int_0^\infty z^k \int_0^\infty \frac{1}{2}t J_0(t\sqrt{z})J_0^3(t)\diff t \diff z.
    \end{align*}
    It is easy to see that $f_T(z):= \frac{1}{2} \int_0^\infty  t J_0(t\sqrt{z})J_0^3(t) \diff t$ is a density with support $[0,9]$ and  $k^{\text{th}}$ moments $\mu_{k}(\mathcal{T}^*)$ for $k\in\N$.
    Hence, $f_T$ is the distribution density of $T$.
    By the same change of variables, both expressions for $f_H$ and $f_T$  via ${}_2F_1$ follow directly from Proposition \ref{pro:density_X_support}.
\end{proof}
Proof of Proposition \ref{prop:charfctHT}:
\begin{proof}
    Let us prove relation \eqref{eq:charFH} by calculating the generating function 
    \begin{align*}
    G_H(x):=\sum_{k=0}^\infty \frac{a_k}{(2k)!}x^{2k}, \quad x\in\C,
    \end{align*}
    of the moment sequence $\mu(\mathcal{H})$ and then setting $x=is$.
    Using the relation
    \begin{align*}
        \int_0^1 t^k (1-t)^k \, \diff t=\frac{(k!)^2}{(2k+1)!}, \quad k\in\N_0,
    \end{align*}
    and  formula (\ref{formula:vladetta}), we write
    \begin{align*}
        G_H(x)
        =
        \sum_{k=0}^\infty \frac{a_k}{(k!)^2} \frac{(k!)^2}{(2k)!} x^{2k} 
        =
        \sum_{k=0}^\infty \frac{a_k}{(k!)^2} \frac{(k!)^2}{(2k+1)!} \frac{\diff}{\diff x} x^{2k+1} 
        &=
        \frac{\diff}{\diff x} \left[ x \int_0^1 \sum_{k=0}^\infty \frac{a_k}{(k!)^2} t^k (1-t)^k x^{2k} \, \diff t \right] \\
        &=
        \frac{\diff}{\diff x} \left[ x \int_0^1 I_0^3 \left( 2x \sqrt{t (1-t)} \right) \, 
        \diff t \right], 
    \end{align*}
    where we exchanged the order of the sum, integral and derivative.
    
    Relation \eqref{eq:charFT} follows from Corollary \ref{prop:generating function}, formula \eqref{eq:generating_function_T} after the substitution $t=-\log y$ in the integral and putting $x=is$.
\end{proof}

\subsection{Approximation of random eigenvalues of $\boldsymbol{T}$}\label{sec:Proofs.subsec:ApproxEigenH} 
Now, we can show our Theorem \ref{theorem:approximation}.

\begin{proof}[Proof of Theorem \ref{theorem:approximation}]
    In order to prove the claim we notice that the support of $3+2Y_b$ is a subset of $[0,9]$ for every $b\in\R$. 
    Hence, the convergence in distribution as $b\to + \infty$ is equivalent to the convergence of moments of $3+2Y_b$ to those of $T$, see e.g., \cite{Kallen02}. 
    For simplicity, let us consider even moments only. The following computation works analogously for odd moments with some changes presented at the end of this proof.
    
    Before we show 
    \begin{align*}
        \E\left(3+2Y_b\right)^{2k}
        \xrightarrow[b\rightarrow + \infty]{}
        \E T^{2k}
        =
        \sum_{k_1+k_2+k_3=2k}\binom{2k}{k_1,k_2,k_3}^2, 
        \quad k\in\N_0,
    \end{align*}
    let us state two observations. 
    
    First, for $k\in\N_0$ recall the well-known formulae
    \begin{align}
        \cos^{k}(x)&=\begin{cases}
        \frac{1}{2^{k}} \binom{k}{k/2}+\frac{1}{2^{k-1}}\sum\limits_{n=0}^{\frac{k}{2}-1} \binom{k}{n} \cos \left( (k-2n)x \right),\quad &\text{ if }k \text{ is even,} \\
        \frac{1}{2^{k-1}}\sum\limits_{n=0}^{\frac{k-1}{2}} \binom{k}{n} \cos \left( \left(k-2n \right)x \right),\quad &\text{ if }k \text{ is odd.} 
        \end{cases}\label{formula:cos^n}
    \end{align}
    
    Second, for $m_1,m_2,m_3\in\N_0$ the following two cases hold.
    
    \textbf{Case 1:} If $m_i\not = m_j$ for (at least) one pair $(i,j)\in\lbrace (1,2),(1,3),(2,3) \rbrace$, we get
    \begin{align}
        &\lim_{b\rightarrow \infty}\frac{1}{b}\int_0^b \cos\left( m_1x \right)\cos\left( m_2\beta x \right)\cos\left( m_3(1+\beta)x \right) \diff x \notag\\
        &=\lim_{b\rightarrow \infty}\frac{1}{4b}\left( 
        \frac{\sin\left( \left( m_1-m_2\beta+m_3(1+\beta) \right)b \right)}{m_1-m_2\beta+m_3(1+\beta)} + 
        \frac{\sin\left( \left( m_1+m_2\beta-m_3(1+\beta) \right)b \right)}{ m_1+m_2\beta-m_3(1+\beta)} \right. \notag \\
        &\quad+
        \left.\frac{\sin\left( \left( m_1-m_2\beta-m_3(1+\beta) \right)b \right)}{m_1-m_2\beta-m_3(1+\beta)} +
        \frac{\sin\left( \left( m_1+m_2\beta+m_3(1+\beta) \right)b \right)}{m_1+m_2\beta+m_3(1+\beta)}
        \right)=0. \label{formula:case1}
    \end{align}
    
    \textbf{Case 2:}  For $m:=m_1=m_2=m_3 \not= 0$, we get
    \begin{align}
        &\lim_{b\rightarrow \infty} \frac{1}{b}\int_0^b \cos\left( m_1x\right)\cos\left( m_2\beta x\right)\cos\left(m_3(1+\beta)x\right) \diff x \notag\\
        &=\lim_{b\rightarrow\infty} \frac{1}{4b}\left( b +\frac{\sin\left( 2mb \right)}{2m}+\frac{\sin\left( 2\beta mb \right)}{2\beta m}+\frac{\sin\left( 2(1+\beta) mb \right)}{2(1+\beta) m} \right)=\frac{1}{4}. \label{formula:case2}
    \end{align} 
    Due to the binomial theorem it holds
    \begin{align*}
        \lim_{b\rightarrow \infty}\E\left(3+2Y_b \right)^{2k} 
        &= \lim_{b\rightarrow\infty}\E\left(\sum_{k_1=0}^{2k} \binom{2k}{k_1}3^{k_1}2^{2k-k_1} Y_b^{2k-k1}\right) =\sum_{k_1=0}^{2k}\binom{2k}{k_1}3^{k_1}2^{2k-k_1} \lim_{b\rightarrow\infty}\E Y_b^{2k-k_1},
    \end{align*}
    for $k\in\N_0$.
    Next, we consider the limit separately. 
    We apply the multinomial theorem to $Y_b^{2k-k_1}$ and get
    \begin{align*}
        \lim_{b\rightarrow\infty}\E Y_b^{2k-k_1} = &\lim_{b\rightarrow\infty}\E\left( \cos X_b+\cos(\beta X_b)+\cos \left((1+\beta) X_b\right)\right)^{2k-k_1}\\
        = &\lim_{b\rightarrow \infty} \E \left( \sum_{k_2=0}^{2k-k_1}\sum_{k_4=0}^{2k-k_1-k_2}\binom{2k-k_1}{k_2,k_4,2k-k_1-k_2-k_4} \right. \\
         & \cdot \left. (\cos X_b)^{k_2} (\cos(\beta X_b))^{k_4} (\cos\left((1+\beta)X_b\right))^{2k-k_1-k_2-k_4} \right)\\
        = &\sum_{k_2=0}^{2k-k_1}\sum_{k_4=0}^{2k-k_1-k_2}\binom{2k-k_1}{k_2,k_4,2k-k_1-k_2-k_4} \\
        & \cdot \lim_{b\rightarrow \infty} \frac{1}{b} \int_0^b (\cos x)^{k_2}(\cos(\beta x))^{k_4} (\cos\left((1+\beta)x\right))^{2k-k_1-k_2-k_4} \diff x.
    \end{align*}
    Due to (\ref{formula:cos^n}), (\ref{formula:case1}), (\ref{formula:case2}) we need to consider both cases when all three $k_1,k_2,k_4$ are either even or odd. Otherwise, the limit of the integral is zero.
    In the case of even $k_j$, it holds
    \begin{align*}
        &\lim_{b\rightarrow \infty}\frac{1}{b}\int_0^b (\cos x)^{k_2}(\cos(\beta x))^{k_4} (\cos\left((1+\beta)x\right))^{2k-k_1-k_2-k_4} \diff x \\ 
        &=\frac{1}{2^{2k-k_1}}\binom{k_2}{k_2/2}\binom{k_4}{k_4/2}\binom{2k-k_1-k_2-k_4}{\frac{2k-k_1-k_2-k_4}{2}}  
        +\frac{1}{2^{2k-k_1-1}}\sum_{n=0}^{\frac{k_2}{2}-1}\binom{k_2}{n}\binom{k_4}{\frac{k_4-k_2}{2}+n}\binom{2k-k_1-k_2-k_4}{\frac{2k-k_1-2k_2-k_4}{2}+n}.
    \end{align*} 
    In the case of odd $k_j$, we get
    \begin{align*}
        &\lim_{b\rightarrow \infty}\frac{1}{b}\int_0^b (\cos x)^{k_2}(\cos(\beta x))^{k_4} (\cos\left((1+\beta)x\right))^{2k-k_1-k_2-k_4} \diff x \\
        &=\frac{1}{2^{2k-k_1-1}}\sum_{n=0}^{\frac{k_2-1}{2}}\binom{k_2}{n}\binom{k_4}{\frac{k_4-k_2}{2}+n}\binom{2k-k_1-k_2-k_4}{\frac{2k-k_1-2k_2-k_4}{2}+n}.
    \end{align*}
    In total, it follows that
    \begin{align*}
        \lim_{b\rightarrow \infty}\E\left(3+2Y_b \right)^{2k} = &\underbrace{\sum_{k_1+k_2+k_3+k_4=k }3^{2k_1}\binom{2k}{2k_1,2k_2,2k_3,2k_4 } \binom{2k_2}{k_2} \binom{2k_3}{k_3} \binom{2k_4}{k_4}}_{=:r_k}\\
        &+\underbrace{2\sum_{k_1+k_2+k_3+k_4=k } 3^{2k_1} \binom{2k}{2k_1,2k_2,2k_3,2k_4} \sum_{n=0}^{k_2-1} \binom{2k_2}{n} \binom{2k_3}{k_3-k_2+n} \binom{2k_4}{k_4-k_2+n}}_{=: s_k}\\
        &+\underbrace{2 \sum_{k_1+k_2+k_3+k_4=k-2 }3^{2k_1+1} \binom{{\scriptstyle 2k}}{{\scriptstyle 2k_1+1},{\scriptstyle 2k_2+1},{\scriptstyle 2k_3+1},{\scriptstyle 2k_4+1}} \sum_{n=0}^{k_2} \binom{{\scriptstyle 2k_2+1}}{{\scriptstyle n}}\binom{{\scriptstyle 2k_3+1}}{{\scriptstyle k_3-k_2+n}}\binom{{\scriptstyle 2k_4+1}}{{\scriptstyle k_4-k_2+n}}}_{=: t_k},
    \end{align*}
    where these sums run over indices $k_1,\ldots, k_4\in \N_0$: $k_1+\ldots+k_4=k$ or $k-2$. 
    Here, we tacitly use the convention that the binomial coefficient $\binom{a}{b}=0$ whenever $a=0$ or $b<0$. 
    
    Next, we compute the generating functions of the sequences $r_k,s_k,t_k$ separately.
    Note that
    \begin{align*}
        \sum_{k=0}^\infty \frac{ 3^{2k}x^{k} }{(2k)!}= \cosh\left(3\sqrt{x}\right) \quad \text{ and }\quad
        \sum_{k=0}^\infty \frac{x^{k}}{(k!)^2} = I_0\left(2\sqrt{x}\right).
    \end{align*}
    Hence, one gets
    \begin{align*}
        H_1(x) := \sum_{k=0}^\infty \frac{x^k}{(2k)!} r_k = \cosh\left(3\sqrt{x}\right)I_0^3\left(2\sqrt{x}\right).
    \end{align*}
    Further, recall that for $a:=k_2-n$ it holds
    \begin{align*}
        \sum_{k=|a|}^\infty \frac{x^{k}}{(k-a)!(k+a)!} = I_{2a}\left(2\sqrt{x}\right) \quad \text{and}\quad \sum_{n=0}^\infty \frac{x^{n}}{n!(n+2a)!} = I_{2a}\left(2\sqrt{x}\right)x^{-a}.
    \end{align*}
    This yields
    \begin{align*}
        H_2(x):= \sum_{k=0}^\infty \frac{x^k}{(2k)!} s_k 
        =&~ 2\cosh(3\sqrt{x})\sum_{k_2=0}^\infty\sum_{n=0}^{k_2-1}\frac{x^{k_2}}{n!(2k_2-n)!}I_{2(k_2-n)}^2\left(2\sqrt{x}\right) \\
        =&~ 2\cosh(3\sqrt{x})\sum_{n=0}^\infty\sum_{k_2=n+1}^\infty\frac{x^{k_2}}{n!(2k_2-n)!}I_{2(k_2-n)}^2\left(2\sqrt{x}\right) \\
        =&~ 2\cosh\left(3\sqrt{x}\right)\sum_{a=1}^\infty\left(\sum_{n=0}^\infty\frac{x^{n}}{n!(n+2a)!}\right)x^{a}I_{2a}^2(2\sqrt{x}) =~ 2\cosh\left(3\sqrt{x}\right)\sum_{a=1}^\infty I_{2a}^3\left(2\sqrt{x}\right).
    \end{align*}
    Next define $a:= 2(k_2-n)+1$, and note that
    \begin{align*}
        \sum_{k_1=0}^\infty \frac{3^{2k_1+1}x^{k_1}}{(2k_1+1)!} = \sinh(3\sqrt{x})x^{-1/2},\quad \sum_{k_3=\max \lbrace a,-a-1\rbrace}^\infty \frac{x^{k_3}}{(k_3-a)!(k_3+a+1)!} = I_{2a-1}(2\sqrt{x})x^{-1/2},
    \end{align*}
    as well as
    \begin{align*}
        \sum_{n=0}^\infty \frac{1}{n!(n+a)!}x^{n} = I_{a}(2\sqrt{x})x^{-a/2}.
    \end{align*}
    Hence, we get
    \begin{align*}
        H_3(x):= \sum_{k=0}^\infty \frac{x^k}{(2k)!} t_k
        =&~ 2\sinh(3\sqrt{x})x^{-1/2}\sum_{k_2=0}^\infty\sum_{n=0}^{k_2}\frac{x^{k_2}}{n!(2k_2-n+1)!}I^2_{2(k_2-n)-1}(2\sqrt{x}) x \\
        =&~ 2\sinh(3\sqrt{x})x^{-1/2}\sum_{n=0}^\infty\sum_{k_2=n}^\infty\frac{x^{k_2}}{n!(2k_2-n+1)!}I^2_{2(k_2-n)-1}(2\sqrt{x}) x \\
        =&~ 2\sinh(3\sqrt{x})x^{-1/2}\sum_{k_2'=0}^\infty\left(\sum_{n=0}^\infty\frac{x^{n}}{n!(n+2k_2'+1)!}\right)x^{k_2'}I^2_{2k_2'+1}(2\sqrt{x}) x \\
        =&~ 2\sinh(3\sqrt{x})\sum_{k_2'=0}^\infty I^3_{2k_2'+1}(2\sqrt{x}),
    \end{align*}
    where $k_2':=k_2-n$.
    Finally, we have
    \begin{align*}
        H(x) =& H_1(x) + H_2(x) + H_3(x) \\
        =& \cosh\left(3\sqrt{x}\right) I_{0}^3\left(2\sqrt{x}\right) + 2\cosh\left(3\sqrt{x}\right)\sum_{k=1}^\infty I_{2k}^3\left(2\sqrt{x}\right) + 2\sinh\left(3\sqrt{x}\right)\sum_{k=1}^\infty I_{2k-1}^3\left(2\sqrt{x}\right) \\
        =& \cosh\left(3\sqrt{x}\right)\sum_{n\in\Z} I_{2n}^3\left(2\sqrt{x}\right)
        + \sinh\left(3\sqrt{x}\right)\sum_{n\in\Z} I_{2n+1}^3\left(2\sqrt{x}\right).
    \end{align*}
    It remains to prove that $\sum_{k=0}^\infty \frac{x^k}{(2k)!}a_{2k} = H(x)$, or, equivalently,
    \begin{align*}
        \sum_{k=0}^\infty \frac{x^{2k}}{(2k)!}a_{2k} = H\left(x^2\right).
    \end{align*}
    Using representation (\ref{eq:generating_function}) from Corollary \ref{prop:generating function} and Theorem \ref{theorem:Besselfunction}
    yields
    \begin{align*}
        \sum_{k=0}^\infty \frac{x^{2k}}{(2k)!}a_{2k} 
        &= \frac{1}{2}\int_0^\infty I_0^3\left(2 \sqrt{xt} \right) e^{-t}\diff t + \frac{1}{2}\int_0^\infty I_0^3\left(2 \sqrt{-xt} \right) e^{-t}\diff t
        \overset{(\ref{identity1})}{=} \frac{1}{2}e^{3x}\sum_{n\in\Z}I_n^3\left(2x\right)+\frac{1}{2}e^{-3x}\sum_{n\in\Z}I_n^3\left(-2x\right)\\
        &= \frac{1}{2}e^{3x}\sum_{n\in\Z}I_{2n}^3\left( 2x\right) + \frac{1}{2}e^{3x}\sum_{n\in\Z}I_{2n+1}^3\left( 2x\right) + \frac{1}{2}e^{-3x}\sum_{n\in\Z}\underbrace{I_{2n}^3\left(- 2x\right)}_{=I_{2n}^3\left( 2x \right)} + \frac{1}{2}e^{-3x}\sum_{n\in\Z} \underbrace{I_{2n+1}^3\left(- 2x\right)}_{=-I_{2n+1}^3\left(2x\right)}\\
        &=\underbrace{\frac{1}{2}\left(e^{3x}-e^{-3x}\right)}_{=\sinh\left(3x\right)} \sum_{n\in\Z} I_{2n}^3\left( 2x\right) + \underbrace{\frac{1}{2}\left(e^{3x}+e^{-3x}\right)}_{= \cosh\left( 3x \right)} \sum_{n\in\Z} I_{2n+1}^3\left( 2x\right) = H\left(x^2\right),
    \end{align*}
    which finishes the proof of convergence in \eqref{eq:approxT}. 
    Analogously, one gets in the odd case
    \begin{align*}
        \E\left(3+2Y_b\right)^{2k+1}
        \xrightarrow[b\rightarrow\infty]{}
        \E T^{2k+1}
        =
        \sum_{k_1+k_2+k_3=2k+1}\binom{2k+1}{k_1,k_2,k_3}^2, 
        \quad k\in\N_0.
    \end{align*} 
    Due to (\ref{formula:cos^n}), (\ref{formula:case1}), (\ref{formula:case2}), it follows that
    \begin{align*}
        \lim_{b\rightarrow \infty}\E\left(3+2Y_b \right)^{2k+1} 
        =& 
        \underbrace{\sum_{k_1=0}^k \sum_{k_2=0}^{k-k_1}\sum_{k_4=0}^{k_3}3^{2k_1+1}\frac{2k+1}{2k_1+1}\binom{k}{k_1,k_2,k_4,k_3-k_4}^2 \binom{2k}{k}\binom{2k_1}{k_1}^{-1}}_{=: \widetilde{r}_k}\\
        &+\underbrace{2 \sum_{k_1=0}^{k}\sum_{k_2=0}^{k-k_1}\sum_{k_4=0}^{k_3-1}\sum_{n=0}^{k_2}3^{2k_1}\binom{{\scriptstyle 2k+1}}{{\scriptstyle 2k_1}}\binom{{\scriptstyle 2(k-k_1)+1}}{{\scriptstyle 2k_2+1,2k_4+1,2(k_3-k_4)-1}}\binom{{\scriptstyle 2k_2+1}}{{\scriptstyle n}}\binom{{\scriptstyle 2k_4+1}}{{\scriptstyle k_4-k_2+n}}\binom{{\scriptstyle 2(k_3-k_4)-1}}{{\scriptstyle k_3-k_2-k_4-1+n}}}_{=: \widetilde{s}_k}\\
        &+\underbrace{2 \sum_{k_1=0}^{k}\sum_{k_2=0}^{k-k_1}\sum_{k_4=0}^{k_3}\sum_{n=0}^{k_2-1}3^{2k_1+1}\binom{{\scriptstyle 2k+1}}{{\scriptstyle 2k_1+1}}\binom{{\scriptstyle 2(k-k_1)}}{{\scriptstyle 2k_2,2k_4,2(k_3-k_4)}}\binom{{\scriptstyle 2k_2}}{{\scriptstyle n}}\binom{{\scriptstyle 2k_4}}{{\scriptstyle k_4-k_2+n}}\binom{{\scriptstyle 2(k_3-k_4)}}{{\scriptstyle k_3-k_2-k_4+n}}}_{=: \widetilde{t}_k},
    \end{align*}
    where $k_3:=k-k_1-k_2$. 
    The generating functions of $\widetilde{r}_k,\widetilde{s}_k,\widetilde{t}_k$ are given as
    \begin{align*}
        \widetilde{H}_1(x) 
        &:= 
        \sum_{k=0}^\infty \frac{x^{k+1/2}}{(2k+1)!} \widetilde{r}_k = \sinh\left(3\sqrt{x}\right)I_0^3\left(2\sqrt{x}\right),\\
        \widetilde{H}_2(x) 
        &:= 
        \sum_{k=0}^\infty \frac{x^{k+1/2}}{(2k+1)!} \widetilde{s_k} 
        =~ 
        2\cosh\left(3\sqrt{x}\right)\sum_{a=0}^\infty I_{2a+1}^3\left(2\sqrt{x}\right),\\
        \widetilde{H}_3(x) 
        &:= 
        \sum_{k=0}^\infty \frac{x^{k+1/2}}{(2k+1)!} \widetilde{t}_k 
        =~ 
        2\sinh\left(3\sqrt{x}\right)\sum_{n=0}^\infty I_{2n}^3(2\sqrt{x}).
    \end{align*}
    Finally, we have
    \begin{align*}
        \widetilde{H}(x) :=& \widetilde{H}_1(x) + \widetilde{H}_2(x) + \widetilde{H}_3(x) 
        = \sinh\left(3\sqrt{x}\right)\sum_{n\in\Z} I_{2n}^3\left(2\sqrt{x}\right)
        + \cosh\left(3\sqrt{x}\right)\sum_{n\in\Z} I_{2n+1}^3\left(2\sqrt{x}\right).
    \end{align*}
    It remains to prove that
    \begin{align*}
        \sum_{k=0}^\infty \frac{x^{2k+1}}{(2k+1)!}a_{2k+1} = \widetilde{H}\left(x^2\right).
    \end{align*}
    Formula (\ref{eq:generating_function_T_odd}) and the identity (\ref{identity1})
    yield
    \begin{align*}
        \sum_{k=0}^\infty \frac{x^{2k+1}}{(2k+1)!}a_{2k+1} 
        &= \frac{1}{2}\int_0^\infty I_0^3\left(2 \sqrt{xt} \right) e^{-t}\diff t - \frac{1}{2}\int_0^\infty I_0^3\left(2 \sqrt{-xt} \right) e^{-t}\diff t\\
        &=\underbrace{\frac{1}{2}\left(e^{3x}-e^{-3x}\right)}_{=\sinh\left(3x\right)} \sum_{n\in\Z} I_{2n}^3\left( 2x\right) + \underbrace{\frac{1}{2}\left(e^{3x}+e^{-3x}\right)}_{= \cosh\left( 3x \right)} \sum_{n\in\Z} I_{2n+1}^3\left( 2x\right) = \widetilde{H}\left(x^2\right),
    \end{align*}
    which finishes the proof of the odd case.
    
    Approximation \eqref{eq:approxH} follows from \eqref{eq:approxT} by the continuous mapping theorem taking distributional equality \eqref{eq:TdH} into account.
\end{proof}
\begin{proof}[Proof of Proposition \ref{prop:1}]
    Interpret independent random variables $U_1,U_2\sim{U}[0,2\pi]$  as polar angles on a unit circle. Then convergence \eqref{weyl} is equivalent to
    \begin{equation}\label{eq:ConvCircle}
        \left( e^{iX_b},e^{i \beta X_b} \right) 
        \xrightarrow[b\rightarrow\infty]{d}
        \left( e^{i U_1},e^{i U_2} \right).
    \end{equation}
     Taking the real and the imaginary parts of both sides of \eqref{eq:ConvCircle} leads to showing
    \begin{align*}
        (\cos(X_b),\sin(X_b),\cos(\beta X_b),\sin(\beta X_b))
        \xrightarrow[b\rightarrow\infty]{d}
        (\cos U_1,\sin U_1,\cos U_2,\sin U_2),
    \end{align*}
    or equivalently
    \begin{align*}
        \lim_{b\to\infty}\Phi_b(s,t,u,v) = I_0(\sqrt{s^2+t^2})I_0(\sqrt{u^2+v^2}), \quad s,t,u,v\in\R,
    \end{align*}
    where 
    \begin{align*}
        \Phi_b(s,t,u,v) 
        :&= 
        \E[\exp \{ s\cos(X_b) +t\sin(X_b)+u\cos(\beta X_b) +v\sin(\beta X_b) \}],\\
        I_0(\sqrt{s^2+t^2}) 
        &= 
        \E[\exp\{s\cos U_1 +t\sin U_1\}]
    \end{align*}
    are the corresponding characteristic functions.
    Now we check the convergence of moments.
    Noting that
    \begin{align*}
        \Phi_b(s,t,u,v) 
        = 
        \sum_{j,k,l,m\ge 0}\frac{s^j t^k u^l v^m}{j!k!l!m!}
        \E[\cos^j(X_b) \sin^k(X_b) \cos^l(\beta X_b) \sin^m(\beta X_b)]
    \end{align*}
    and
    \begin{align*}
       I_0(\sqrt{s^2+t^2})
        =& 
        \sum_{\substack{m\ge 0\\ m\text{ even}}}\frac{(s^2+t^2)^{m/2}}{(m/2)!(m/2)!2^{m}} = \sum_{\substack{m\ge 0\\ m \text{ even}}}\sum_{\substack{j+k=m\\ j,k \text{ even}}}\binom{m/2}{j/2}\frac{s^j t^k}{(m/2)!(m/2)!2^{m}} \\
        =& 
        \sum_{\substack{j,k\ge 0\\ j,k\text{ even}}} \frac{s^j t^k}{((j+k)/2)!(j/2)!(k/2)!2^{j+k}},
    \end{align*}
    it suffices to show that
    \begin{align*}
        & \lim_{b\to\infty}\E[\cos^j(X_b) \sin^k(X_b) \cos^l(\beta X_b) \sin^m(\beta X_b)] \\
        &= \begin{cases}
            \displaystyle
            \frac{ j! k! l! m!}{(j+k)/2)!(j/2)!(k/2)!((l+m)/2)!(l/2)!(m/2)!2^{j+k+l+m}}, & j,k,l,m  \quad {\rm even}, \\
            0, & \mathrm{otherwise}.
        \end{cases}
    \end{align*}
    This can be confirmed using relations \eqref{formula:cos^n}--\eqref{formula:case2} similar to the proof of Theorem \ref{theorem:approximation}.
\end{proof}

\section{Summary}\label{sec:summary}
This paper investigates the distribution of a randomly chosen eigenvalue $H$ ($T$, respectively) of the adjacency matrix of two infinite regular lattices on the plane: hexagonal lattice (graphene) and its dual, the triangulation of the plane. 
Explicit formulae for the probability densities, moment generating functions, and characteristic functions have been obtained. 
A connection to symmetric random walks on these lattices as well as planar random flights was established. 
We presented a direct simulation approach for generating random eigenvalues $T$ and $H$. 
Additionally, we provide an approximation theorem, which approximates these eigenvalues by a simple function involving the sum of three cosines of a uniformly distributed random variable with irrational frequencies. 
It was also shown how this approximation is related to the ergodic theory. 
As a side effect of this research, a new identity for the series of third powers of modified Bessel functions $I_n(\cdot)$, $n\in\Z$, was proven. 
This series can be thus expressed as an integral of $I_0^3(\cdot)$.

\subsection*{Acknowledgments}
The authors express their sincere gratitude to Christoph Koutschan and Nobuki Takayama for their invaluable insights and critical feedback on Theorem \ref{theorem:Besselfunction} regarding the cubic Bessel identity.
They also extend their appreciation to Mr. Minahan for his detailed feedback.

\bigskip

\bibliography{Literature}{}

\begin{thebibliography}{10}

\bibitem{abert2013benjamini}
M.~Ab{\'e}rt, A.~Thom, and B.~Vir{\'a}g.
\newblock Benjamini-{S}chramm convergence and pointwise convergence of the
  spectral measure.
\newblock preprint at
  \url{https://citeseerx.ist.psu.edu/document?repid=rep1&type=pdf&doi=1edcf72b450386f49dfeb74751399691434788a7},
  2013.

\bibitem{abra}
M.~Abramowitz and I.~A. Stegun.
\newblock {\em Handbook of Mathematical Functions with Formulas, Graphs, and
  Mathematical Tables}.
\newblock Dover, 1964.

\bibitem{ahlfors78}
L.~Ahlfors.
\newblock {\em Complex Analysis: An Introduction to the Theory of Analytic
  Functions of One Complex Variable}.
\newblock International Series in Pure and Applied Mathematics. McGraw-Hill
  Book Co., New York, third edition, 1978.

\bibitem{albenque2021local}
M.~Albenque, L.~M{\'e}nard, and G.~Schaeffer.
\newblock Local convergence of large random triangulations coupled with an
  ising model.
\newblock {\em Transactions of the American Mathematical Society},
  374:175--217, 2021.

\bibitem{Aldous2007}
D.~Aldous and R.~Lyons.
\newblock Processes on unimodular random networks.
\newblock {\em Electronic Journal of Probability}, 12:1454--1508, 2007.

\bibitem{aldous2004objective}
D.~Aldous and J.~Steele.
\newblock The objective method: probabilistic combinatorial optimization and
  local weak convergence.
\newblock {\em Probability on {D}iscrete {S}tructures}, 110:1--72, 2004.

\bibitem{angel2018isoperimetric}
O.~Angel, I.~Benjamini, and N.~Horesh.
\newblock An isoperimetric inequality for planar triangulations.
\newblock {\em Discrete \& Computational Geometry}, 59:802--809, 2018.

\bibitem{angel2003uniform}
O.~Angel and O.~Schramm.
\newblock Uniform infinite planar triangulations.
\newblock {\em Communications in Mathematical Physics}, 241:191--213, 2003.

\bibitem{bordenave2016spectrum}
F.~Benaych-Georges, C.~Bordenave, M.~Capitaine, C.~Donati-Martin, and
  A.~Knowles.
\newblock Spectrum of random graphs.
\newblock In {\em Advanced Topics in Random Matrices}, volume~53, pages
  91--150. Societe Mathematique de France, 2017.

\bibitem{benjaminicurien12}
I.~Benjamini and N.~Curien.
\newblock {Ergodic theory on stationary random graphs}.
\newblock {\em Electronic Journal of Probability}, 17:1--20, 2012.

\bibitem{benjamini2011recurrence}
I.~Benjamini and O.~Schramm.
\newblock Recurrence of distributional limits of finite planar graphs.
\newblock {\em Selected Works of Oded Schramm}, 6:533--545, 2011.

\bibitem{RsoftBille23}
A.~Bille.
\newblock Python software for the simulation of random eigenvalues of graphenes
  and the triangulation of plane.
\newblock \url{https://github.com/fullereneuulm}, 2024.

\bibitem{Bille2020SpectralCO}
A.~Bille, V.~Buchstaber, and E.~Spodarev.
\newblock Spectral clustering of combinatorial fullerene isomers based on their
  facet graph structure.
\newblock {\em Journal of Mathematical Chemistry}, 59:264--288, 2020.

\bibitem{Blevins1997}
R.~Blevins.
\newblock Probability density of finite {F}ourier series with random phases.
\newblock {\em Journal of Sound and Vibration}, 208:617--652, 1997.

\bibitem{borwein15}
J.~M. Borwein, A.~Straub, and C.~Vignat.
\newblock Densities of short uniform random walks in higher dimensions.
\newblock {\em Journal of Mathematical Analysis and Applications},
  437:668--707, 2016.

\bibitem{borwein2012densities}
J.~M. Borwein, A.~Straub, J.~Wan, W.~Zudilin, and D.~Zagier.
\newblock Densities of short uniform random walks.
\newblock {\em Canadian Journal of Mathematics}, 64:961--990, 2012.

\bibitem{buckygen}
G.~Brinkmann, J.~Goedgebeur, and B.~McKay.
\newblock The generation of fullerenes.
\newblock {\em Journal of Chemical Information and Modeling}, 52:2910--2918,
  2012.

\bibitem{BuchEro17}
V.~Buchstaber and N.~Erokhovets.
\newblock Finite sets of operations sufficient to construct any fullerene from
  {$C_{20}$}.
\newblock {\em Structural Chemistry}, 28:225--234, 2017.

\bibitem{Cotfas00}
N.~Cotfas.
\newblock Random walks on carbon nanotubes and quasicrystals.
\newblock {\em Journal of Physics A}, 33:2917--2927, 2000.

\bibitem{coxeter1973regular}
H.~S.~M. Coxeter.
\newblock {\em Regular Polytopes}.
\newblock Courier Corporation, 1973.

\bibitem{DiCre}
A.~Di~Crescenzo, C.~Macci, B.~Martinucci, and S.~Spina.
\newblock {Analysis of random walks on a hexagonal lattice}.
\newblock {\em IMA Journal of Applied Mathematics}, 84:1061--1081, 2019.

\bibitem{diestel_graphtheory}
R.~Diestel.
\newblock {\em Graph Theory}.
\newblock Springer, New York, 2005.

\bibitem{engel2018number}
P.~Engel and P.~Smillie.
\newblock The number of convex tilings of the sphere by triangles, squares, or
  hexagons.
\newblock {\em Geometry \& Topology}, 22:2839--2864, 2018.

\bibitem{granville20}
A.~Granville.
\newblock {\em Number Theory Revealed: A Masterclass}.
\newblock American Mathematical Society, 2020.

\bibitem{grunbaum1987tilings}
B.~Gr{\"u}nbaum and G.~C. Shephard.
\newblock {\em Tilings and Patterns}.
\newblock Courier Dover Publications, 1987.

\bibitem{horiguchi1972lattice}
T.~Horiguchi.
\newblock Lattice {G}reen's functions for the triangular and honeycomb
  lattices.
\newblock {\em Journal of Mathematical Physics}, 13:1411--1419, 1972.

\bibitem{Kallen02}
O.~Kallenberg.
\newblock {\em Foundations of Modern Probability}.
\newblock Probability and Its Applications. Springer, 2002.

\bibitem{kasteleyn1967graph}
P.~Kasteleyn.
\newblock Graph theory and crystal physics.
\newblock In {\em Graph Theory and Theoretical Physics}, pages 43--110.
  Academic Press, 1967.

\bibitem{koutchan2009}
C.~Koutschan.
\newblock Advanced applications of the holonomic systems approach.
\newblock {\em ACM Communications in Computer Algebra}, 43:119, 2010.

\bibitem{koutchan2010}
C.~Koutschan.
\newblock Holonomic{F}unctions (user’s guide).
\newblock Technical Report in RISC Report Series, University of Linz, Austria,
  \url{https://www3.risc.jku.at/publications}, 2010.

\bibitem{kuipers_niederreiter}
L.~Kuipers and H.~Niederreiter.
\newblock {\em Uniform Distribution of Sequences}.
\newblock Pure and Applied Mathematics. Wiley, first edition, 1974.

\bibitem{wondra}
U.~M{\"u}ller and H.~Wondratschek.
\newblock {\em International Tables for Crystallography: Volume A1 Symmetry
  Relations Between Space Groups}.
\newblock Kluwer Academic Publishers, 2004.

\bibitem{Rayleigh1880}
L.~Rayleigh.
\newblock On the resultant of a large number of vibrations of the same pitch
  and of arbitrary phase.
\newblock {\em The London, Edinburgh, and Dublin Philosophical Magazine and
  Journal of Science}, 10:73--78, 1880.

\bibitem{Rukh18}
A.~D. Rukhovich.
\newblock On the growth rate of the number of fullerenes.
\newblock {\em Russian Mathematical Surveys}, 73:734--736, 2018.

\bibitem{Sloane10}
N.~Sloane.
\newblock On-{L}ine {E}ncyclopedia of {I}nteger {S}equences.
\newblock \url{http://oeis.org}, 2010.

\bibitem{Thurs98}
W.~P. Thurston.
\newblock Shapes of polyhedra and triangulations of the sphere.
\newblock In {\em The {E}pstein Birthday Schrift}, volume~1 of {\em Geometry \&
  Topology Monographs}, pages 511--549. University of Warwick Mathematics
  Institute, 1998.

\bibitem{trott18}
M.~Trott.
\newblock A tale of three cosines -- an experimental mathematics adventure.
\newblock Wolfram blog. Accessed: 02.06.2023.
  \url{https://blog.wolfram.com/2018/04/24/a-tale-of-three-cosines-an-experimental-mathematics-adventure},
  April 2018.

\bibitem{watson66}
G.~Watson.
\newblock {\em A Treatise on the Theory of Bessel Functions}.
\newblock Cambridge University Press, 1966.

\bibitem{weyl1916}
H.~Weyl.
\newblock {\"U}ber die {G}leichverteilung von {Z}ahlen mod. {E}ins.
\newblock {\em Mathematische Annalen}, 77:313--352, 1916.

\end{thebibliography}
\bibliographystyle{abbrv}

\end{document}